\documentclass{amsart}%
\usepackage{amsfonts}
\usepackage{amsmath}
\usepackage{amssymb}
\usepackage{graphicx}%
\providecommand{\U}[1]{\protect\rule{.1in}{.1in}}
\newtheorem{theorem}{Theorem}
\theoremstyle{plain}

\newtheorem{corollary}{Corollary}

\newtheorem{definition}{Definition}
\newtheorem{example}{Example}

\newtheorem{lemma}{Lemma}

\newtheorem{proposition}{Proposition}
\newtheorem{remark}{Remark}

\numberwithin{equation}{section}
\begin{document}
\title[Semi $r$-ideals of commutative rings]{Semi $r$-ideals of commutative rings}
\author{Hani A. Khashan}
\address{Department of Mathematics, Faculty of Science, Al al-Bayt University, Al
Mafraq, Jordan.}
\email{hakhashan@aabu.edu.jo}
\author{Ece YETKIN\ CELIKEL}
\address{Department of Basic Sciences, Faculty of Engineering, Hasan Kalyoncu
University, Gaziantep, Turkey.}
\email{ece.celikel@hku.edu.tr, yetkinece@gmail.com}
\thanks{This paper is in final form and no version of it will be submitted for
publication elsewhere.}

\begin{abstract}
For commutative rings with identity, we introduce and study the concept of
semi $r$-ideals which is a kind of generalization of both $r$-ideals and
semiprime ideals. A proper ideal $I$ of a commutative ring $R$ is called semi
$r$-ideal if whenever $a^{2}\in I$ and $Ann_{R}(a)=0$, then $a\in I$. Several
properties and characterizations of this class of ideals are determined. In
particular, we investigate semi $r$-ideal under various contexts of
constructions such as direct products, localizations, homomorphic images,
idealizations and amalagamations rings. We extend semi $r$-ideals of rings to
semi $r$-submodules of modules and clarify some of their properties. Moreover,
we define submodules satisfying the $D$-annihilator condition and justify when
they are semi $r$-submodules.

\end{abstract}
\maketitle

\section{Introduction}

Throughout, all rings are supposed to be commutative with identity and all
modules are unital. Let $R$ be a ring and $M$ an $R$-module. We recall that a
proper ideal $I$ of a $R$ is called semiprime if whenever $a\in R$ such that
$a^{2}\in I$, then $a\in I$. It is well-known that $I$ is semiprime in $R$ if
and only if $I$ is a radical ideal, that is $I=\sqrt{I}$ where $\sqrt{I}=$
$\{x\in R:x^{m}\in I$ for some $m\in%
\mathbb{Z}
\}$. In 2015, R. Mohamadian \cite{Moh} introduced the concept of $r$-ideals of
commutative rings. A proper ideal $I$ of a ring $R$ is called an $r$-ideal
(resp. $pr$ -ideal) if whenever $a,b\in R$ such that $ab\in I$ and
$Ann_{R}(a)=0$, then $b\in I$ (resp. $b\in\sqrt{I}$) where $Ann_{R}%
(a)=\left\{  b\in R:ab=0\right\}  $. Prime and $r$-ideals are not comparable
in general;\ but it is verified that every maximal $r$-ideal in a ring is a
prime ideal, while every minimal prime ideal is an $r$ -ideal. In 2017, Tekir,
Koc and Oral \cite{Tek} introduced the concept of $n$-ideals as a special kind
of $r$-ideals by considering the set of nilpotent elements instead of zero
divisors. Recently, in \cite{Haniece3}, Celikel and Khashan generalized
$n$-ideals by defining and studying the class of semi $n$-ideals. A proper
ideal $I$ of $R$ is called a semi $n$-ideal if for $a\in R$, $a^{2}\in I$ and
$a\notin\sqrt{0}$ imply $a\in I$. Later, some other generalizations of
semiprime, $n$-ideals and $r$-ideals have been introduced, see for
example,\cite{Badawi}, \cite{Hani}-\cite{Haniece2} and \cite{Ece}.

Motivated by semiprime ideals and semi $n$-ideals, we define a proper ideal
$I$ of a ring $R$ to be a semi $r$-ideal if whenever $a\in R$ such that
$a^{2}\in I$ and $Ann_{R}(a)=0$, then $a\in I$. It is clear that the class of
semi $r$-ideals is a generalization of that of semiprime and $r$-ideals. We
start section 2 by giving some examples (see Example \ref{ex1}) to show that
this generalization is proper. Next, we determine several equivalent
characterizations of semi $r$-ideals (see Theorem \ref{char}). Among many
other results in this paper, we characterize rings in which every ideal is a
semi $r$-ideal (see Theorem \ref{Every}). We investigate semi $r$-ideals under
various contexts of constructions such as homomorphic images, quotient rings,
localizations and polynomial rings (see Propositions \ref{f} and \ref{S},
Corollary \ref{quotient}, Theorem \ref{I[x]}). Moreover, we discuss and
characterize semi $r$-ideals of cartesian product of rings (see Proposition
\ref{Ca2}, Theorems \ref{ca1} and \ref{cchar}, Corollaries \ref{cca1} and
\ref{cc}). Let $R$ and $S$ be two rings, $J$ be an ideal of $S$ and
$f:R\rightarrow S$ be a ring homomorphism. We study some forms of semi
$r$-ideals of the amalgamation ring $R\Join^{f}J$ of $R$ with $S$ along $J$
with respect to $f$ (see Theorems \ref{a1} and \ref{a2}).

Let $M$ be an $R$-module, $N$ be a submodule of $M$ and $I$ be an ideal of
$R$. As usual, we will use the notations $(N:_{R}M)$ and $(N:_{M}I)$ for the
sets $\{r\in R:rm\in N$ for all $m\in M\}$ and $\{m\in M:$ $Im\subseteq N\},$
respectively. In particular, the annihilator of an element $m\in M$ (resp.
$r\in R)$ denoted by $Ann_{R}(m)$ (resp. $Ann_{M}(r))$, is $(0:_{R}m)$ (resp.
$(0:_{M}r)$. We recall that the torsion subgroup $T(M)$\ of an $R$-module $M$
is defined as $T(M)=\{m\in M:$ there exists $0\neq r\in R$ such that $rm=0\}.$
It is easy to see that $T(M)$ is a submodule of $M$, called the torsion
submodule. A module is torsion (resp. torsion-free) if $T(M)=M$ (resp.
$T(M)=\{0\}$).

In 2009, the concept of semiprime submodules is presented. A proper submodule
is said to be semiprime if whenever $r\in R$, $m\in M$ and $r^{2}m\in N$, then
$rm\in N,$ \cite{Sar}. Afterwards, the notions of $r$-submodule and
$sr$-submodules are introduced and studied in \cite{Suat}. A proper submodule
$N$ is called an $r$-submodule (resp. $sr$-submodule) of $M$ if whenever
$rm\in N$ and $Ann_{M}(r)=0_{M}$ (resp. $Ann_{R}(m)=0$)$,$ then $m\in N$
(resp. $r\in(N:_{R}M)$). As a new generalization of above structures, in
Section 3, we define a proper submodule $N$ of $M$ to be a semi $r$-submodule
if whenever $r\in R$, $m\in M$ with $r^{2}m\in N$, $Ann_{M}(r)=0_{M}$ and
$Ann_{R}(m)=0$, then $rm\in N$. We illustrate (see Example \ref{ex5}) that
this generalization of $r$-submodules is proper. However, it is observed that
semi $r$-submodules coincides with semiprime submodules in any torsion-free
module. Then, we introduce a new condition for submodules, namely,
$D$-annihilator condition as follows: A proper submodule $N$ of an $R$-module
$M$ is said to satisfy the $D$-annihilator condition if whenever $K$ is a
submodule of $M$ and $r\in R$ such that $rK\subseteq N$ and $Ann_{M}(r)=0_{M}%
$, then either $K\subseteq N$ or $K\cap T(M)=\left\{  0_{M}\right\}  $. By
using this condition, we totally characterize semi $r$-submodules of finitely
generated faithful multiplication $R$-modules (see Proposition \ref{eqM},
Theorems \ref{IM} and \ref{IN}, Corollary \ref{(N:M)}).

We recall that the idealization of an $R$-module $M$ denoted by $R(+)M$, is
the commutative ring $R\times M$ with coordinate-wise addition and
multiplication defined as $(r_{1},m_{1})(r_{2},m_{2})=(r_{1}r_{2},r_{1}%
m_{2}+r_{2}m_{1})$. For an ideal $I$ of $R$ and a submodule $N$ of $M$,
$I(+)N$ is an ideal of $R(+)M$ if and only if $IM\subseteq N$. It is well
known from \cite{Ande} that
\[
zd(R(+)M)=\left\{  (r,m)|\text{ }r\in zd(R)\cup Z(M)\text{, }m\in M\right\}
\]
In Proposition \ref{Ide}, we clarify the relation between semi $r$-ideals of
the idealization ring $R(+)M$ and those of $R$ which enables us to build some
interesting examples of semi $r$-ideals.

Let $f:R_{1}\rightarrow R_{2}$ be a ring homomorphism, $J$ be an ideal of
$R_{2}$, $M_{1}$ be an $R_{1}$-module, $M_{2}$ be an $R_{2}$-module and
$\varphi:M_{1}\rightarrow M_{2}$ be an $R_{1}$-module homomorphism. The
subring
\[
R_{1}\Join^{f}J=\left\{  (r,f(r)+j):r\in R_{1}\text{, }j\in J\right\}
\]
of $R_{1}\times R_{2}$ is called the amalgamation of $R_{1}$ and $R_{2}$ along
$J$ with respect to $f$. In \cite{Rachida}, the amalgamation of $M_{1}$ and
$M_{2}$ along $J$ with respect to $\varphi$ is defined as%

\[
M_{1}\Join^{\varphi}JM_{2}=\left\{  (m_{1},\varphi(m_{1})+m_{2}):m_{1}\in
M_{1}\text{ and }m_{2}\in JM_{2}\right\}
\]
which is an $(R_{1}\Join^{f}J)$-module. The last section is devoted to clarify
semi $r$-submodules of the amalgamation of modules.

\section{Properties of semi $r$-ideals}

This section deals with many properties of semi $r$-ideals. We justify the
relations among the concepts of semiprime ideals, semi $n$-ideals and our new
class of ideals. Moreover, several characterizations and examples are
presented. In particular, we characterize rings in which every ideal is a semi
$r$-ideal.

\begin{definition}
Let $I$ be a proper ideal of a ring $R$. $I$ is called a semi $r$-ideal of $R$
if whenever $a\in R$ such that $a^{2}\in I$ and $Ann_{R}(a)=0$, then $a\in I$.
\end{definition}

For any non-zero subset $A$ of a ring $R$, we note that $Ann_{R}(A)$ is a semi
$r$-ideal of $R$. It is clear that the classes of semiprime ideals, $r$-ideals
and semi $n$-ideals are contained in the class of semi $r$-ideals. However, in
general these containments are proper as we illustrate in the following examples.

\begin{example}
\label{ex1}Let $p$ and $q$ be prime integers.

\begin{enumerate}
\item Any non-zero semiprime ideal in an integral domain is a semi $r$-ideal
that is not an $r$-ideal.

\item In the ring $%
\mathbb{Z}
_{p^{2}q}$, the ideal $\left\langle \overline{p^{2}}\right\rangle $ is a semi
$r$-ideal that is not a semi $n$-ideal.

\item The zero ideal of a ring $R$ is always a semi $r$-ideal but it is not a
semiprime ideal unless $R$ is a semiprime ring.

\item Every ideal of a Boolean ring (a ring of which every element is
idempotent) is semi $r$-ideal. Consider the ideal $I=0\times0\times%
\mathbb{Z}
_{2}$ of the Boolean ring $%
\mathbb{Z}
_{2}\times%
\mathbb{Z}
_{2}\times%
\mathbb{Z}
_{2}$. Then $I$ is a semi $r$-ideal that is not prime.

\item In general pr-ideals and semi r-ideals are not comparable. Let $T$ be a
reduced ring with subring $%
\mathbb{Z}
$ and $P$ be a nonzero minimal prime ideal in $T$ with $P\cap%
\mathbb{Z}
=(0)$. From \cite[Example 2.17]{Moh}, $J=x^{2}P[x]$ is a $pr$ -ideal of the
ring $R=%
\mathbb{Z}
+xT[x]$. Choose an element $0\neq p\in P$. Then $(xp)^{2}\in J$ and
$Ann_{R}(xa)=0$ but $xa\notin J$. Thus, $J$ is not a semi $r$-ideal. Moreover,
any non-zero prime ideal in an integral domain is clearly a semi $r$-ideal
that is not a $pr$-ideal.
\end{enumerate}
\end{example}

If $I$ and $J$ are semi $r$-ideals of a ring $R$, then $IJ$ and $I+J$ need not
be so as we can see in the following example.

\begin{example}
Consider the ideals $I=\left\langle x\right\rangle $ and $J=\left\langle
x-4\right\rangle $ of the ring $R=%
\mathbb{Z}
\lbrack x]$. Then $I$ and $J$ are (semi) prime ideals and so are semi
$r$-ideals of $R$. On the other hand, $I+J=\left\langle x,x-4\right\rangle
=\left\langle x,4\right\rangle $ is not a semi $r$-ideal of $R$. Indeed,
$(2+x)^{2}\in I+J$ and $Ann_{R}(2+x)=0$, but $2+x\notin I+J$. Also,
$I^{2}=\left\langle x^{2}\right\rangle $ is not a semi $r$-ideal of $R$ as
$x^{2}\in I^{2}$ and $Ann_{R}(x)=0$, but $x\notin I^{2}$.
\end{example}

Next, we give the following characterization of semi $r$-ideals. By $zd(R)$ we
denote the set of all zero divisor elements of a ring $R$. Moreover, $reg(R)$
denotes the set $R\backslash zd(R)$.

\begin{theorem}
\label{char}Let $I$ be a proper ideal of a ring $R$ and $k$ be a positive
integer. The following statements are equivalent.
\end{theorem}

\begin{enumerate}
\item $I$ is a semi $r$-ideal of $R.$

\item Whenever $a\in R$ with $0\neq a^{2}\in I$ and $Ann_{R}(a)=0$, then $a\in
I$.

\item Whenever $a\in R$ with $a^{k}\in I$ and $Ann_{R}(a)=0$, then $a\in I$.

\item $\sqrt{I}\subseteq zd(R)\cup I$.
\end{enumerate}

\begin{proof}
(1)$\Leftrightarrow$(2). Suppose $(2)$ holds and let $a\in R$ such that
$a^{2}\in I$ and $Ann_{R}(a)=0$. If $a^{2}=0$, then $a=0$ and the result
follows obviously. If $a^{2}\neq0$, then we are also done by (2). The converse
part is obvious.

(1)$\Rightarrow$(3). Suppose $a^{k}\in I$ and $Ann_{R}(a)=0$ for $a\in R$. We
use the mathematical induction on $k$. If $k\leq2$, then the claim is clear.
We now assume that (3) holds for all $2<t<k$ and show that it is also true for
$k.$ Suppose $k$ is even, say, $k=2m$ for some positive integer $m.$ Since
$a^{k}=(a^{m})^{2}\in I$ and clearly $Ann_{R}(a^{m})=0$, then $a^{m}\in I$ as
$I$ is a semi $r$-ideal. By the induction hypothesis, we conclude that $a\in
I$ as needed. Suppose $k$ is odd, so that $k+1=2s$ for some $s<k$. Then
similarly, we have $\left(  a^{s}\right)  ^{2}\in I$ and $Ann_{R}(a^{s})=0$
which imply that $a^{s}\in I$ and again by the induction hypothesis, we
conclude $a\in I$.

(3)$\Rightarrow$(4). Let $a\in\sqrt{I}$. Then $a^{k}\in I$ for some $k\geq1$
and so by (3) $a\in zd(R)$ or $a\in I$. Thus, $\sqrt{I}\subseteq zd(R)\cup I$.

(4)$\Rightarrow$(1). Straightforward.
\end{proof}

\begin{corollary}
\label{coro}Let $I$ be a semi $r$-ideal of a ring $R$ and $k$ be a positive
integer. If $J$ is an ideal of $R$ with $J^{k}\subseteq I$ and $J\cap
zd(R)=\left\{  0\right\}  $, then $J\subseteq I$.
\end{corollary}

\begin{proof}
Suppose that $J^{k}\subseteq I$ and $J\cap zd(R)=\left\{  0\right\}  $ for
some ideal $J$ of $R.$ Let $0\neq a\in J$. From the assumption $J\cap
zd(R)=\left\{  0\right\}  $, we have $Ann_{R}(a)=0$. Thus, $a^{k}\in I$
implies that $a\in I$ by Theorem \ref{char} (3).
\end{proof}

\begin{corollary}
Let $I$ and $J$ be proper ideals of a ring $R$ such that $I\cap zd(R)=J\cap
zd(R)=\left\{  0\right\}  .$
\end{corollary}

\begin{enumerate}
\item If $I$ and $J$ are semi $r$-ideals of a ring $R$ with $I^{2}=J^{2}$,
then $I=J$.

\item If $I^{2}$ is a semi $r$-ideal, then $I^{2}=I$.
\end{enumerate}

\begin{proof}
(1) Since $I^{2}\subseteq J$ and $J\cap zd(R)=\left\{  0\right\}  $, then we
have $I\subseteq J$ by Corollary \ref{coro}. On the other hand, since
$J^{2}\subseteq I$ and $J\cap zd(R)=\left\{  0\right\}  ,$ we have $J\subseteq
I$ again by Corollary \ref{coro}, so we are done.

(2) A direct consequence of (1).
\end{proof}

We note by example \ref{ex1} that unlike $r$-ideals, if $I$ is a semi
$r$-ideal of a ring $R$, then $I$ need not be contained in $zd(R)$. Also,
clearly, semi $r$-ideals which contain the zero divisors of a ring $R$ are semiprime.

Next, we present a condition for a semi $r$-ideal to be an $r$-ideal. First,
we need the following lemma.

\begin{lemma}
\label{max}Let $S$ be a non-empty subset of $R$ where $S\cap zd(R)=\emptyset$.
If $I$ is a semi $r$-ideal of $R$ with $S\nsubseteq I$, then $(I:S)$ is a semi
$r$-ideal of $R$.
\end{lemma}

\begin{proof}
Let $a\in R$ such that $a^{2}\in(I:S)$ and $Ann_{R}(a)=0$. Then $(as)^{2}\in
I$ for all $s\in S$. As $I$ is a semi $r$-ideal of $R$, we have either $as\in
zd(R)$ or $as\in I$ for all $s\in S$. If $as\in zd(R)$, then $S\cap
zd(R)=\emptyset$ implies $a\in zd(R)$, a contradiction. Thus, $as\in I$ for
all $s\in S$ and so $a\in(I:S)$ as required.
\end{proof}

\begin{theorem}
If $I$ is maximal among all semi $r$-ideals of a ring $R$ contained in
$zd(R)$, then $I$ is an $r$-ideal.
\end{theorem}

\begin{proof}
Let $I$ be maximal among all semi $r$-ideals of a ring $R$ contained in
$zd(R)$. Suppose that $ab\in I$ and $Ann_{R}(a)=0$. Then $a\notin I\cup zd(R)$
and so $(I:_{R}a)$ is a semi $r$-ideal of $R$ by Lemma \ref{max}. Since
clearly, $(I:_{R}a)\subseteq zd(R)$ and $I\subseteq(I:_{R}a)$, then the
maximality of $I$ implies, $I=(I:_{R}a)$. Thus, $b\in I$ and $I$ is an $r$-ideal.
\end{proof}

Following \cite{Moh}, we call a ring $R$ a $uz$-ring if $R=U(R)\cup zd(R)$. It
is proved in \cite{Moh} that $R$ is a $uz$-ring if and only if every ideal in
$R$ is an $r$-ideal. In particular, a direct product of fields is an example
of a $uz$-ring. Next, we generalize this result to semi $r$-ideals.

\begin{theorem}
\label{Every}The following statements are equivalent for a ring $R$.
\end{theorem}

\begin{enumerate}
\item $R$ is a $uz$-ring.

\item Every proper ideal of $R$ is an $r$-ideal.

\item Every proper ideal of $R$ is a semi $r$-ideal.

\item Every proper principal ideal of $R$ is a semi $r$-ideal.

\item Every semi $r$-ideal is an $r$-ideal.
\end{enumerate}

\begin{proof}
(1)$\Rightarrow$(2). Follows by \cite[Proposition 3.4]{Moh}.

(2)$\Rightarrow$(3)$\Rightarrow$(4). Clear.

(4)$\Rightarrow$(1). Let $x\in R\backslash zd(R)$. If $\left\langle
x^{2}\right\rangle =R$, then $x\in U(R)$. Suppose $\left\langle x^{2}%
\right\rangle $ is proper in $R$. Since $x^{2}\in\left\langle x^{2}%
\right\rangle $ and $Ann_{R}(x)=0$ , then by assumption, $x\in\left\langle
x^{2}\right\rangle $. Thus, $x=rx^{2}$ for some $r\in R$ and so $rx=1$ as
$Ann_{R}(x)=0$. Thus, again $x\in U(R)$ and $R=U(R)\cup zd(R)$ as needed.

(1)$\Rightarrow$(5). Clear by (1)$\Leftrightarrow$(2).

(5)$\Rightarrow$(1). Since a maximal ideal of $R$ is clearly a semi $r$-ideal,
then by (5), every maximal ideal in $R$ is an $r$-ideal. Let $r\in R$. If
$r\notin U(R)$, then $r\in M$ for some maximal ideal $M$ of $R$ and so $r\in
zd(R)$ by \cite[Remark 2.3(d)]{Moh}. Therefore, $R=U(R)\cup zd(R)$ and $R$ is
a $uz$-ring.
\end{proof}

Next, we discuss the behavior of semi $r$-ideals under homomorphisms.

\begin{proposition}
\label{f}Let $f:R_{1}\rightarrow R_{2}$ be a ring homomorphism. The following
statements hold.
\end{proposition}

\begin{enumerate}
\item If $f$ is an epimorphism, $I_{1}\subseteq Ker(f)$ and $I_{1}$ is a semi
$r$-ideal of $R_{1}$ such that $I_{1}\cap zd(R_{1})=\left\{  0\right\}  $,
then $f(I_{1})$ is a semi $r$-ideal of $R_{2}$.

\item If $f$ is an isomorphism and $I_{2}$ is a semi $r$-ideal of $R_{2}$,
then $f^{-1}(I_{2})$ is a semi $r$-ideal of $R_{1}.$
\end{enumerate}

\begin{proof}
(1) Let $a\in R_{2}$ such that $a^{2}\in f(I_{1})$ and $a\notin f(I_{1})$.
Then there exists $x\in R_{1}\backslash I_{1}$ such that $a=f(x)$. Since
$f(x^{2})=a^{2}\in f(I_{1})$, then $x^{2}\in I_{1}$ as $Ker(f)\subseteq I_{1}%
$. Now, $I_{1}$ is a semi $r$-ideal of $R_{1}$ implies $x\in zd(R_{1})$. If
$x=0$, then $a=f(x)\in zd(R_{2})$. Suppose $x\neq0$ and choose $0\neq y\in R$
such that $xy=0$. Then $f(y)\neq0$ since otherwise $y\in I_{1}\cap zd(R_{1})$,
a contradiction. Thus, again $a=f(x)\in zd(R_{2})$ and $f(I_{1})$ is a semi
$r$-ideal of $R_{2}$.

(2) Suppose $I_{2}$ is a semi $r$-ideal of $R_{2}$. Let $x\in R_{1}$ such that
$x^{2}\in f^{-1}(I_{2})$ and $x\notin f^{-1}(I_{2})$. Then $f(x^{2}%
)=f(x)^{2}\in I_{2}$ and $f(x)\notin I_{2}$ which imply $f(x)\in zd(R_{2})$.
Since $f$ is an isomorphism, then clearly $x\in zd(R_{1})$ and $f^{-1}(I_{2})$
is a semi $r$-ideal of $R_{1}.$
\end{proof}

In view of Proposition \ref{f}, we have the following result for quotient rings.

\begin{corollary}
\label{quotient}Let $I$ and $J$ be ideals of a ring $R$ with $J\subseteq I$.
\end{corollary}

\begin{enumerate}
\item If $I$ is a semi $r$-ideal of $R$ and $I\cap zd(R)=\left\{  0\right\}
$, then $I/J$ is a semi $r$-ideal of $R/J$.

\item If $I/J$ is a semi $r$-ideal of $R/J$ and $J$ is an $r$-ideal of $R$,
then $I$ is a semi $r$-ideal of $R.$
\end{enumerate}

\begin{proof}
(1). Consider the natural epimorphism $\pi:R\rightarrow R/J$ with $Ker(\pi)=J$
and apply Proposition \ref{f}.

(2). Let $a\in R$ such that $a^{2}\in I$ and $a\notin zd(R)$. Then
$(a+J)^{2}=a^{2}+J\in I/J$. If $a+J\in zd(R/I)$, then there is $b\notin J$
such that $ab\in J$. Since $J$ is a semi $r$-ideal of $R$, we get $a\in
zd(R)$, a contradiction. Thus, $a+J\notin zd(R/I)$ which yields $a+J\in I/J$
as $I/J$ is a semi $n$-ideal of $R/J$ and so $a\in I$.
\end{proof}

If $I\cap zd(R)\neq\left\{  0\right\}  $ in Corollary \ref{quotient}(1), then
the result need not be true. For example, $4%
\mathbb{Z}
(+)%
\mathbb{Z}
_{4}$ is a semi $r$-ideal of $%
\mathbb{Z}
(+)%
\mathbb{Z}
_{4}$, see Remark \ref{Ide}. But $4%
\mathbb{Z}
(+)%
\mathbb{Z}
_{4}/0(+)%
\mathbb{Z}
_{4}\cong4%
\mathbb{Z}
$ is not a semi $r$-ideal of $%
\mathbb{Z}
(+)%
\mathbb{Z}
_{4}/0(+)%
\mathbb{Z}
_{4}\cong%
\mathbb{Z}
$. We also note that the condition " $J$ is an $r$-ideal" in Corollary
\ref{quotient}(2) is crucial. For example $8%
\mathbb{Z}
/16%
\mathbb{Z}
$ is a semi $r$-ideal of $%
\mathbb{Z}
/16%
\mathbb{Z}
$ but $8%
\mathbb{Z}
$ is not a semi $r$-ideal of $%
\mathbb{Z}
$.

In particular, Corollary \ref{quotient} holds if $J\subseteq zd(R)$.

\begin{proposition}
\label{inters}The intersection of any family of semi $r$-ideals is a semi $r$-ideal.
\end{proposition}

\begin{proof}
Let $\left\{  I_{\alpha}:\alpha\in\Lambda\right\}  $ is a family of semi
$r$-ideals. Suppose $a^{2}\in\bigcap\limits_{\alpha\in\Lambda}I_{\alpha}$ and
$a\notin\bigcap\limits_{\alpha\in\Lambda}I_{\alpha}$. Then $a\notin I_{\gamma
}$ for some $\gamma\in\Lambda$. Since $I_{\gamma}$ is a semi $r$-ideal, we
have $a\in zd(R)$ and so $\bigcap\limits_{\alpha\in\Lambda}I_{\alpha}$ is a
semi $r$-ideal.
\end{proof}

Let $I$ be a proper ideal of $R.$ In the following we give the relationship
between semi $r$-ideals of a ring and those of its localization ring by using
the notation $Z_{I}(R)$\ which denotes the set $\{r\in R$ $|$ $rs\in I$ for
some $s\in R\backslash I\}$.

\begin{proposition}
\label{S}Let $S$ be a multiplicatively closed subset of a ring $R$ such that
$S\cap zd(R)=\emptyset$. Then the following hold.
\end{proposition}

\begin{enumerate}
\item If $I$ is a semi $r$-ideal of $R$ such that $I\cap S=\emptyset$, then
$S^{-1}I$ is a semi $r$-ideal of $S^{-1}R.$

\item If $S^{-1}I$ is a semi $r$-ideal of $S^{-1}R$ and $S\cap Z_{I}%
(R)=\emptyset$, then $I$ is a semi $r$-ideal of $R.$
\end{enumerate}

\begin{proof}
(1)\ Suppose for $\frac{a}{s}\in S^{-1}R$ that $\left(  \frac{a}{s}\right)
^{2}\in S^{-1}I$ \ and $\left(  \frac{a}{s}\right)  \notin S^{-1}I$. Then
there exits $u\in S$ such that $ua^{2}\in I$ and so $(ua)^{2}\in I$. Since
clearly $ua\notin I$ and $I$ is a semi $r$-ideal, we have $ua\in zd(R)$, say,
$(ua)b=0$ for some $0\neq b\in R$. Thus, $\frac{a}{s}\cdot\frac{b}{1}%
=\frac{uab}{us}=0_{S^{-1}R}$ and $\frac{b}{1}\neq0_{S^{-1}R}$ as $S\cap
zd(R)=\emptyset$. Thus, $\frac{a}{s}\in zd(S^{-1}R)$ and $S^{-1}I$ is a semi
$r$-ideal of $S^{-1}R.$

(2) Suppose $a^{2}\in I$ for $a\in R$. Since $S^{-1}I$ is a semi $n$-ideal of
$S^{-1}R$ and $\left(  \frac{a}{1}\right)  ^{2}\in S^{-1}I$, we have either
$\frac{a}{1}\in S^{-1}I$ or $\frac{a}{1}\in zd(S^{-1}R).$ If $\frac{a}{1}\in
S^{-1}I$, then there exists $u\in S$ such that $ua\in I$. Since $S\cap
zd(R)=\emptyset,$ we conclude that $a\in I.$ If $\frac{a}{1}\in zd(S^{-1}R),$
then there is $\frac{b}{t}\neq0_{S^{-1}R}$ such that $\frac{ab}{t}=\frac{a}%
{1}\cdot\frac{b}{t}=0_{S^{-1}R}$. Hence, $vab=0$ for some $v\in S$ and so
$ab=0$ as $S\cap zd(R)=\emptyset$. Thus, $a\in zd(R)$ as $b\neq0$ and $I$ is a
semi $r$-ideal of $R.$
\end{proof}

We recall that if $f=\sum\limits_{i=1}^{m}a_{i}x^{i}\in R[x]$, then the ideal
$\left\langle a_{1},a_{2},\cdots,a_{m}\right\rangle $ of $R$ generated by the
coefficients of $f$ is called the content of $f$ and is denoted by $c(f)$. It
is well known that if $f$ and $g$ are two polynomials in $R[x]$, then the
content formula $c(g)^{m+1}c(f)=c(g)^{m}c(fg)$ holds where $m$ is the degree
of $f$ , \cite[Theorem 28.1]{Gilmer}. For an ideal $I$ of $R$, it can be
easily seen that $I[x]=\left\{  f(x)\in R[x]:c(f)\subseteq I\right\}  $.

\begin{definition}
A ring $R$ is said to satisfy the property ($\ast$) if whenever $f\in
reg(R[x])$, then $c(f)\backslash\left\{  0\right\}  \subseteq reg(R)$.
\end{definition}

\begin{theorem}
\label{I[x]}Let $I$ be an ideal of a ring $R$.
\end{theorem}

\begin{enumerate}
\item If $I[x]$ is a semi $r$-ideal of $R[x]$, then $I$ is a semi $r$-ideal of
$R$.

\item If $R$ satisfies the property ($\ast$) and $I$ is a semi $r$-ideal of
$R$, then $I[x]$ is a semi $r$-ideal of $R[x]$
\end{enumerate}

\begin{proof}
(1) Suppose $I[x]$ is a semi $r$-ideal of $R[x]$. Let $a\in R$ such that
$a^{2}\in I$ and $Ann_{R}(a)=0$. Then Clearly, $a^{2}\in I[x]$ and
$Ann_{R[x]}(a)=0$. By assumption, $a\in I[x]$ and so $a\in I$ as required.

(2) Suppose $R$ satisfies the property ($\ast$) and $I$ is a semi $r$-ideal of
$R$. Let $f(x)\in R[x]$ such that $\left(  f(x)\right)  ^{2}\in I[x]$ and
$Ann_{R[x]}(f(x))=0$. Then $c(f^{2})\subseteq I$ and so by the content
formula, $(c(f))^{2}=c(f^{2})\subseteq I$. Moreover, $c(f)\cap zd(R)=\left\{
0\right\}  $ as $R$ satisfies the property ($\ast$) and so $c(f)\subseteq I$
by Corollary \ref{coro}. It follows that $f(x)\in I[x]$ and we are done.
\end{proof}

In general, if $S$ is an overring of a ring $R$, then we may find a semi
$r$-ideal $J$ of $S$ where $J\cap R$ is not a semi $r$-ideal in $R$.

\begin{example}
Let $S=%
\mathbb{Z}
\times%
\mathbb{Z}
$ and consider the ring homomorphism $\varphi:%
\mathbb{Z}
\longrightarrow%
\mathbb{Z}
\times%
\mathbb{Z}
$ defined by $\varphi(x)=(x,0)$. Then $\varphi$ is a monomorphism and so
$R=\varphi(%
\mathbb{Z}
)$ is a domain. Now, $J=Ann_{S}((0,1))$ is a nonzero (semi) $r$-ideal in $S$.
However, clearly, $R\subseteq J$ and so $J\cap R=R$ is not a semi $r$-ideal in
$R$.
\end{example}

Let $S$ be an overring ring of a ring $R$ . Following \cite{Moh}, $R$ is said
to be essential in $S$ if $J\cap R\neq\left\{  0\right\}  $ for every nonzero
ideal $J$ of $S$ .

\begin{proposition}
Let $R\subseteq S$ be rings such that $R$ is essential in $S$. If $J$ is a
semi $r$ -ideal of $S$, then $J\cap R$ is a semi $r$-ideal in $R$.
\end{proposition}

\begin{proof}
Let $a\in R$ such that $a^{2}\in J\cap R$ and $Ann_{R}(a)=0$. Then $a\in S$
with $a^{2}\in J$ and $Ann_{S}(a)=0$. Indeed, if $Ann_{S}(a)\neq0$, then $R$
being essential implies $Ann_{S}(a)\cap R\neq\left\{  0\right\}  $. Thus,
there exists $0\neq r\in R$ such that $r\in Ann_{S}(a)$ and so $r\in
Ann_{R}(a)$, a contradiction. Since $J$ is a semi $r$ -ideal of $S$, then
$a\in J\cap R$ and the result follows.,
\end{proof}

The rest of this section is devoted to discuss semi $r$-ideals of cartesian
products of rings and their particular subrings: the amalgamation rings.

\begin{proposition}
\label{Ca2}Let $R=R_{1}\times R_{2}$ where $R_{1}$ and $R_{2}$ are two rings
and $I_{1}$, $I_{2}$ be proper ideals of $R_{1}$ and $R_{2}$, respectively.
Then $I_{1}\times R_{2}$ (resp. $R_{1}\times I_{2}$) is a semi $r$-ideal of
$R$ if and only if $I_{1}$ is a semi $r$-ideal of $R_{1}$ (resp. $I_{2}$ is a
semi $r$-ideal of $R_{2}$)$.$
\end{proposition}

\begin{proof}
Let $I_{1}\times R_{2}$ be a semi $r$-ideal of $R$ and $a\in R_{1}$ with
$a^{2}\in I_{1}$ and $Ann_{R_{1}}(a)=0.$ Then $(a,1)^{2}\in$ $I_{1}\times
R_{2}$ and $Ann_{R}(a,1)=(0,0)$ imply that $(a,1)\in I_{1}\times R_{2}$ and so
$a\in I_{1}$. Thus $I_{1}$ is a semi $r$-ideal of $R_{1}.$ Conversely, suppose
that $(a,b)^{2}\in I_{1}\times R_{2}$ and $Ann_{R}(a,b)=(0,0).$ Then $a^{2}\in
I_{1}$ and clearly $Ann_{R_{1}}(a)=0$ which implies $a\in I_{1}$. Hence,
$(a,b)\in I_{1}\times R_{2},$ so we are done. The proof of the case
$R_{1}\times I_{2}$ is similar.
\end{proof}

The following corollary generalizes Proposition \ref{Ca2}.

\begin{corollary}
\label{cca1}Let $R_{1},R_{2},$ $\cdots,R_{n}$ be rings, $R=R_{1}\times
R_{2}\times\cdots\times R_{n}$ and $I_{i}$ be a proper ideal of $R_{i}$ for
each $i=1,2,\cdots n$. Then for all $j=1,2,\cdots n$, $I=R_{1}\times
\cdots\times R_{j-1}\times I_{j}\times R_{j+1}\times\cdots\times R_{n}$ is a
semi $r$-ideal of $R$ if and only if $I_{j}$ is a semi $r$-ideal of $R_{j}$.
\end{corollary}

\begin{theorem}
\label{ca1}Let $R_{1}$ and $R_{2}$ be two rings, $R=R_{1}\times R_{2}$ and
$I_{1},I_{2}$ be proper ideals in $R_{1}$ and $R_{2}$, respectively.
\end{theorem}

\begin{enumerate}
\item If $I_{1}$ and $I_{2}$ are semi $r$-ideals of $R_{1}$ and $R_{2}$,
respectively, then $I=I_{1}\times I_{2}$ is a semi $r$-ideal of $R$.

\item If $I=I_{1}\times I_{2}$ is a semi $r$-ideal of $R$, then either $I_{1}$
is a semi $r$-ideal of $R_{1}$ or $I_{2}$ is a semi $r$-ideal of $R_{2}$.

\item If $I=I_{1}\times I_{2}$ is a semi $r$-ideal of $R$ and $I_{2}\nsubseteq
zd(R_{2})$, then $I_{1}$ is a semi $r$-ideal of~$R_{1}.$

\item If $I=I_{1}\times I_{2}$ is a semi $r$-ideal of $R$ and $I_{1}\nsubseteq
zd(R_{1})$, then $I_{2}$ is a semi $r$-ideal of~$R_{2}.$
\end{enumerate}

\begin{proof}
(1) Let $(a,b)\in R$ such that $(a^{2},b^{2})=(a,b)^{2}\in I$ and
$Ann_{R}(a,b)=(0,0)$. Then $a^{2}\in I_{1}$, $b^{2}\in I_{2}$ and clearly
$Ann_{R_{1}}(a)=Ann_{R_{2}}(b)=0$. Therefore, $a\in I_{1}$, $b\in I_{2}$ and
so $(a,b)\in I$ as needed.

(2).Suppose $I=I_{1}\times I_{2}$ is a semi $r$-ideal of $R$ but $I_{1}$ and
$I_{2}$ are not semi $r$-ideals of $R_{1}$ and $R_{2}$, respectively. Choose
$a\in R_{1}$ and $b\in R_{2}$ such that $a^{2}\in I_{1}$, $b^{2}\in I_{2}$,
$Ann_{R1}(a)=0$ and $Ann_{R_{2}}(b)=0$ but $a\notin I_{1}$ and $b\notin I_{2}%
$. Then $(a,b)^{2}\in I$ and clearly, $Ann_{R}(a,b)=(0,0)$. By assumption, we
have $(a,b)\in I$ which is a contradiction. Therefore, either $I_{1}$ is a
semi $r$-ideal of $R_{1}$ or $I_{2}$ is a semi $r$-ideal of $R_{2}$.

(3) Suppose $a^{2}\in I_{1}$ for some $a\in R_{1}$ with $Ann_{R_{1}}(a)=0$.
Since $I_{2}\nsubseteq Z(R_{2})$, we can choose $b\in I_{2}\cap reg(R_{2})$.
Then $(a,b)^{2}\in I$ and $Ann_{R}(a,b)=(0,0).$ It follows that $(a,b)\in
I$;\ and hence $a\in I_{1}.$

(4) is similar to (3).
\end{proof}

The converse of Theorem \ref{ca1}(1) is not true in general. For example, $4%
\mathbb{Z}
\times0$ is a semi $r$-ideal in $%
\mathbb{Z}
\times%
\mathbb{Z}
$ by Proposition \ref{red}. On the other hand, the ideal $4%
\mathbb{Z}
$ is not a semi $r$-ideals of $%
\mathbb{Z}
$.

The following corollary generalizes Theorem \ref{ca1} to any finite direct
product of rings. The proof is similar to that of Theorem \ref{ca1}.

\begin{corollary}
\label{cc}Let $R_{1},R_{2},$ $\cdots,R_{n}$ be rings, $R=R_{1}\times
R_{2}\times\cdots\times R_{n}$ and $I_{i}$ be a proper ideal of $R_{i}$ for
each $i=1,2,\cdots n$.
\end{corollary}

\begin{enumerate}
\item If $I_{i}$ is a semi $r$-ideals of $R_{i}$ for each $i=1,2,\cdots n$,
then $I=I_{1}\times I_{2}\times\cdots\times I_{n}$ is a semi $r$-ideal of $R$.

\item If $I=I_{1}\times I_{2}\times\cdots\times I_{n}$ is a semi $r$-ideal of
$R$, then $I_{j}$ is a semi $r$-ideal of $R_{j}$ for at least one
$j\in\left\{  1,2,\cdots,n\right\}  $.

\item If $I=I_{1}\times I_{2}\times\cdots\times I_{n}$ is a semi $r$-ideal of
$R$ and $I_{j}\nsubseteq Z(R_{j})$ for all $j\neq i$, then $I_{i}$ is a semi
$r$-ideal of~$R_{i}.$
\end{enumerate}

\begin{lemma}
\label{red}Let $R=R_{1}\times R_{2}\times\cdots\times R_{n}$ where $R_{i}$'s
are rings and $R_{j}$ is reduced ring for some $j=1,...,n$. If $I_{i}$ is an
ideal of $R_{i}$ for all $i\neq j$, then $I=I_{1}\times\cdots\times
I_{j-1}\times0\times I_{j+1}\times\cdots\times I_{n}$ is a semi $r$-ideal of
$R$.
\end{lemma}

\begin{proof}
Let $a=$ $(a_{1},a_{2},...,a_{n})\in R$ with $a^{2}\in I$. Then $a_{j}^{2}=0$
which implies $a_{j}=0$ as $R_{j}$ is reduced. Since $Ann_{R}(a)=Ann_{R}%
(a_{1},...,a_{j-1},0,a_{j+1},...,a_{n})\neq0$, $I$ is a semi $r$-ideal of $R$.
\end{proof}

Next, we present a characterization for semi $r$-ideals of cartesian products
of domains.

\begin{theorem}
\label{cchar}Let $R_{1},R_{2},$ $\cdots,R_{n}$ $(n\geq2)$ be domains,
$R=R_{1}\times R_{2}\times\cdots\times R_{n}$ and $I_{i}$ be an ideal of
$R_{i}$ for each $i=1,2,\cdots n$. Then $I=I_{1}\times I_{2}\times\cdots\times
I_{n}$ is a semi $r$-ideal of $R$ if and only if one of the following
statements holds
\end{theorem}

\begin{enumerate}
\item $I_{j}=\{0\}$ for at least one $j\in\left\{  1,2,\cdots,n\right\}  $.

\item There exists $j\in\left\{  1,2,\cdots n\right\}  $ such that $I_{i}$ is
a semi $r$-ideal of $R_{i}$ for all $i=1,\cdots,j$ and $I_{i}=$ $R_{i}$ for
all $i=j+1,\cdots,n$.

\item $I_{i}$ is a semi $r$-ideals of $R_{i}$ for each $i=1,2,\cdots n$.
\end{enumerate}

\begin{proof}
Suppose $I=I_{1}\times I_{2}\times\cdots\times I_{n}$ is a semi $r$-ideal of
$R.$ Suppose that all $I_{i}$'s are nonzero. If for all $i\in\left\{
1,2,\cdots n\right\}  $, $I_{i}$ is proper in $R_{i}$, then $I_{i}$ is a semi
$r$-ideals of $R_{i}$ by Corollary \ref{cc}(3). Without loss of generality
assume that $I_{1},...,I_{j}$ are proper in $R_{1},\cdots,R_{j}$, respectively
and $I_{i}=R_{i}$ for all $i\in\{j+1,...,n\}.$ For each $i\in\{2,...,j\}$,
choose a nonzero element $b_{i}\in I_{i}.$ Let $a\in R_{1}$ such that
$a^{2}\in I_{1}$. Since $(a,b_{2},b_{3},...b_{j},1_{R_{j+1}},...,1_{R_{n}%
})^{2}\in I$ and $Ann_{R}(a,b_{2},b_{3},...b_{j},1_{R_{j+1}},...,1_{R_{n}%
})=0,$ we have $(a,b_{2},b_{3},...b_{j},1_{R_{j+1}},...,1_{R_{n}})\in I$ and
so $a\in I_{1}$. Therefore, $I_{1}$ is a semi $r$-ideal of $R_{1}.$ Similarly,
$I_{i}$ is a semi $r$-ideals of $R_{i}$ for all $i\in\{1,...,j\}$.

Conversely, if (1) holds, then $I$ is clearly a semi $r$-ideal of $R$. Suppose
that $I_{1},...,I_{j}$ are semi $r$-ideals and $I_{k}=R_{k}$ for all
$k\in\{j+1,...,n\}$. Let $a=(a_{1},a_{2},...,a_{n})\in R$ with $a^{2}\in I$
and $Ann_{R}(a)=0.$ Then for each $i\in\{1,...,j\}$, $a_{i}^{2}\in I$ and
$Ann_{R_{i}}(a_{i})=0$ as $R_{i}$'s are domain. Thus, $a_{i}\in I_{i}$ and so
$a\in I$. Finally, if (3)\ holds, then $I=I_{1}\times I_{2}\times\cdots\times
I_{n}$ is a semi $r$-ideal of $R$ by Corollary \ref{cc}(1).
\end{proof}

Let $R$ and $S$ be two rings, $J$ be an ideal of $S$ and $f:R\rightarrow S$ be
a ring homomorphism. As a subring of $R\times S,$ the amalgamation of $R$ and
$S$ along $J$ with respect to $f$ is defined by $R\Join^{f}J=(a,f(a)+j):a\in
R$, $j\in J\}.$ If $f$ is the identity homomorphism on $R$, then we get the
amalgamated duplication of $R$ along an ideal $J$, $R\Join J=\left\{
(a,a+j):a\in R\text{, }j\in J\right\}  $. For more related definitions and
several properties of this kind of rings, one can see \cite{DAnna1}. If $I$ is
an ideal of $R$ and $K$ is an ideal of $f(R)+J$, then $I\Join^{f}J=\left\{
(i,f(i)+j):i\in I\text{, }j\in J\right\}  $ and $\bar{K}^{f}=\{(a,f(a)+j):a\in
R$, $j\in J$, $f(a)+j\in K\}$ are ideals of $R\Join^{f}J$, \cite{DAnna3}.

\begin{lemma}
\label{am}\cite{Azimi} Let $R$, $S$, $J$ and $f$ be as above. Let
$A=\{(r,f(r)+j)|r\in zd(R)\}$ and $B=\{(r,f(r)+j)|j^{\prime}(f(r)+j)=0$ for
some $j^{\prime}\in J\backslash\{0\}\}$. Then $zd(R\Join^{f}J)\subseteq A\cup
B$.
\end{lemma}

Next, we determine conditions under which $I\Join^{f}J$ and $\bar{K}^{f}$ are
semi $r$-ideals of $R\Join^{f}J$.

\begin{theorem}
\label{a1}Let $R$, $S$, $J$ and $f$ be as above. If $I$ is a semi $r$-ideal of
$R$, then $I\Join^{f}J$ is a semi $r$-ideal of $R\Join^{f}J$. The converse is
true if $f(reg(R))\cap Z(J)=\emptyset$
\end{theorem}

\begin{proof}
Suppose $I$ is a semi $r$-ideal of $R$. Let $(a,f(a)+j)\in R\Join^{f}J$ such
that $(a,f(a)+j)^{2}=(a^{2},f(a^{2})+2jf(a)+j^{2})\in I\Join^{f}J$ and
$(a,f(a)+j)\notin zd(R\Join^{f}J)$. Then $a^{2}\in I$ and $a\notin zd(R)$ by
Lemma \ref{am}. Therefore, $a\in I$ and so $(a,f(a)+j)\in I\Join^{f}J$ as
needed. Now, suppose $f(reg(R))\cap Z(J)=\emptyset$ and $I\Join^{f}J$ is a
semi $r$-ideal of $R\Join^{f}J$. Let $a^{2}\in I$ for $a\in R$ and $a\notin
zd(R)$. Then $(a,f(a))\in R\Join^{f}J$ with $(a,f(a))^{2}=(a^{2},f(a^{2}))\in
I\Join^{f}J$. If $(a,f(a))\in zd(R\Join^{f}J)$, then Lemma \ref{am} implies
$f(a)\in Z(J)$ which is a contradiction. Therefore, $(a,f(a))\notin
zd(R\Join^{f}J)$ and so $(a,f(a))\in I\Join^{f}J$ as $I\Join^{f}J$ is a semi
$r$-ideal of $R\Join^{f}J$. Thus, $a\in I$ as required.
\end{proof}

\begin{theorem}
\label{a2}Let $f:R\rightarrow S$ be a ring homomorphism and $J,K$ be ideals of
$S$. If $K$ is a semi $r$-ideal of $f(R)+J$, then $\bar{K}^{f}$ is a semi
$r$-ideal of $R\Join^{f}J$.
\end{theorem}

\begin{enumerate}
\item If $K$ is a semi $r$-ideal of $f(R)+J$ and $zd(f(R)+J)=Z(J)$, then
$\bar{K}^{f}$ is a semi $r$-ideal of $R\Join^{f}J$.

\item If $\bar{K}^{f}$ is a semi $r$-ideal of $R\Join^{f}J$,
$f(zd(R))\subseteq zd(f(R)+J)$ and $f(zd(R))J=0$, then $K$ is a semi $r$-ideal
of $f(R)+J$.
\end{enumerate}

\begin{proof}
(1) Suppose $K$ is a semi $r$-ideal of $f(R)+J$. Let $(a,f(a)+j)\in R\Join
^{f}J$ such that $(a,f(a)+j)^{2}=(a^{2},(f(a)+j)^{2})\in\bar{K}^{f}$ and
$(a,f(a)+j)\notin zd(R\Join^{f}J)$. Then $(f(a)+j)^{2}\in K$ and by Lemma
\ref{am}, $f(a)+j\notin Z(J)=zd(f(R)+J)$. Therefore, $f(a)+j\in K$ and
$(a,f(a)+j)\in\bar{K}^{f}$ as needed.

(2) Suppose $\bar{K}^{f}$ is a semi $r$-ideal of $R\Join^{f}J$ and
$f(zd(R))J=0$. Let $f(a)+j\in f(R)+J$ such that $(f(a)+j)^{2}\in K$ and
$f(a)+j\notin zd(f(R)+J)$. Then $(a,f(a)+j)\in$ $R\Join^{f}J$ with
$(a,f(a)+j)^{2}\in\bar{K}^{f}$. Suppose $(a,f(a)+j)\in zd(R\Join^{f}J)$. Then
as $Z(J)\subseteq zd(f(R)+J)$ and by Lemma \ref{am}, we conclude that $a\in
zd(R)$. Since $f(a)\in zd(f(R)+J)$, then $f(a)f(b)=0$ for some $0\neq f(b)\in
f(R)$. Thus, $(f(a)+j)f(b)=0$ as $f(zd(R))J=0$ which contradicts that
$f(a)+j\notin zd(f(R)+J)$. Therefore, $(a,f(a)+j)\notin zd(R\Join^{f}J)$ and
so $(a,f(a)+j)\in\bar{K}^{f}$. It follows that $f(a)+j\in K$ and $K$ is a semi
$r$-ideal of $f(R)+J$.
\end{proof}

\section{Semi $r$-submodules of modules over commutative rings}

The aim of this section is to extend semi $r$-ideals of commutative rings to
semi $r$-submodules of modules over commutative rings. Recall that a module
$M$ is said to be faithful if $Ann_{R}(M)=(0:_{R}M)=0_{R}$.

\begin{definition}
Let $M$ be an $R$-module and $N$ a proper submodule of $M.$
\end{definition}

\begin{enumerate}
\item $N$ is called a semiprime submodule if whenever $r^{2}m\in N$, then
$rm\in N.$ \cite{Sar}

\item $N$ is called a $r$-submodule if whenever $rm\in N$ and $Ann_{M}%
(r)=0_{M},$ then $m\in N.$ \cite{Suat}

\item $N$ is called a $sr$-submodule if whenever $rm\in N$ and $Ann_{R}(m)=0,$
then $m\in N.$ \cite{Suat}
\end{enumerate}

\begin{definition}
Let $M$ be an $R$-module and $N$ a proper submodule of $M$. We call $N$ a semi
$r$-submodule if whenever $r\in R$, $m\in M$ with $r^{2}m\in N$,
$Ann_{M}(r)=0_{M}$ and $Ann_{R}(m)=0$, then $rm\in N$.
\end{definition}

The reader clearly observe that any semi $r$-submodule of an $R$-module $R$ is
a semi $r$-ideal of $R$. The zero submodule is always a semi $r$-submodule of
$M.$ Also, see the implications:

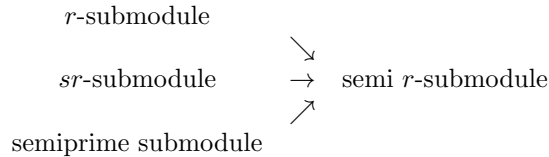
\begin{center}
$%
\begin{array}
[c]{ccc}%
r\text{-submodule} &  & \\
& \searrow & \\
sr\text{-submodule} & \rightarrow & \text{semi }r\text{-submodule}\\
& \nearrow & \\
\text{semiprime submodule} &  &
\end{array}
$
\end{center}

\bigskip

However, the next examples show that these arrows are irreversible.

\begin{example}
\label{ex5}
\end{example}

\begin{enumerate}
\item Consider the submodule $N=6%
\mathbb{Z}
\times\left\langle 0\right\rangle $ of the $%
\mathbb{Z}
$-module $M=%
\mathbb{Z}
\times%
\mathbb{Z}
.$ Let $r\in%
\mathbb{Z}
$ and $m=(m_{1},m_{2})\in M$ such that $r^{2}\cdot(m_{1},m_{2})\in N.$ Then
$r^{2}m_{1}\in6%
\mathbb{Z}
$, $r^{2}m_{2}=0$ and $Ann_{%
\mathbb{Z}
}(r)=Ann_{%
\mathbb{Z}
}(m_{1})=Ann_{%
\mathbb{Z}
}(m_{2})=0$ as $%
\mathbb{Z}
$ is a domain. Since $6%
\mathbb{Z}
$ and $\left\langle 0\right\rangle $ are semi $r$-ideals of $%
\mathbb{Z}
$, then $r\cdot(m_{1},m_{2})\in N$ and so $N$ is a semi $r$-submodule of $M$.
On the other hand, we have $2\cdot(3,0)\in N$ with $Ann_{M}(2)=0_{M}$ and
$Ann_{%
\mathbb{Z}
}((3,0))=0$ but $(3,0)\notin N$ and so $N$ is neither $r$-submodule nor
$sr$-submodule of $M$.

\item Consider the submodule $N=\left\langle \bar{4}\right\rangle
\times\left\langle 0\right\rangle $ of the $%
\mathbb{Z}
$-module $M=%
\mathbb{Z}
_{8}\times%
\mathbb{Z}
$. Let $r\in%
\mathbb{Z}
$ and $m=(m_{1},m_{2})\in M$ such that $r^{2}\cdot(m_{1},m_{2})\in N.$ Then it
is clear to observe that $Ann_{%
\mathbb{Z}
}(r)=Ann_{%
\mathbb{Z}
}(m_{1})=Ann_{%
\mathbb{Z}
}(m_{2})=0.$ Since again $N$ is a semi $r$-submodule of $M$ as $\left\langle
\bar{4}\right\rangle $ is a semi $r$-ideal of $%
\mathbb{Z}
_{8}$ and $\left\langle 0\right\rangle $ is a semi $r$-ideals of $%
\mathbb{Z}
$. However, $2^{2}\cdot(\bar{1},0)\in N$ but $2\cdot(\bar{1},0)\notin N$ and
so $N$ is not a semiprime submodule of $M$.
\end{enumerate}

\begin{proposition}
Let $M$ be an $R$-module, $N$ a proper submodule of $M$ and $k$ any positive
integer. Then $N$ is a semi $r$-submodule of $M$ if and only if whenever $r\in
R$, $m\in M$ with $r^{k}m\in N$, $Ann_{M}(r)=0_{M}$ and $Ann_{R}(m)=0$, then
$rm\in N$.
\end{proposition}

\begin{proof}
The proof follows by mathematical induction on $k$ in a similar way to that of
Theorem \ref{char} (3).
\end{proof}

We recall that a module $M$ is torsion (resp. torsion-free) if $T(M)=M$ (resp.
$T(M)=\{0\}$) where $T(M)=\{m\in M:$ there exists $0\neq r\in R$ such that
$rm=0\}$. It is clear that any torsion-free module is faithful.

\begin{proposition}
Semi $r$-submodules and semiprime submodules are coincide in any torsion-free module.
\end{proposition}

\begin{proof}
Since every semiprime submodule is semi $r$-submodule, we need to show the
converse. Let~$N$ be a semi $r$-submodule of an $R$-module $M$, $r\in R$,
$m\in M$ with $r^{2}m\in N$. Keeping in mind that $M$ is torsion-free, we have
$Ann_{R}(m)=0$. Now, suppose that $m^{\prime}\in Ann_{M}(r).$ Then
$rm^{\prime}=0$ and if $r=0$, then clearly $rm\in N$. If $r\neq0$, then
$m^{\prime}=0$ again as $M$ is torsion-free. Since $N$ is a semi
$r$-submodule, we conclude $rm\in N$, as required.
\end{proof}

\begin{definition}
A proper submodule $N$ of an $R$-module $M$ is said to satisfy the
$D$-annihilator condition if whenever $K$ is a submodule of $M$ and $r\in R$
such that $rK\subseteq N$ and $Ann_{M}(r)=0_{M}$, then either $K\subseteq N$
or $K\cap T(M)=\left\{  0_{M}\right\}  $.
\end{definition}

Obviously, any $r$-submodule satisfies the $D$-annihilator condition. The
converse is not true in general. For example the submodule $N=6%
\mathbb{Z}
\times\left\langle 0\right\rangle $ of the $%
\mathbb{Z}
$-module $M=%
\mathbb{Z}
\times%
\mathbb{Z}
$ clearly satisfies the $D$-annihilator condition. On the other hand, $N$ is
not an $r$-submodule of $M$, (see Example \ref{ex5}(1)). It is clear that any
proper submodule of a torsion-free module satisfies the $D$-annihilator
condition. However, we may find a submodule satisfying the $D$-annihilator
condition in a torsion module. For example, for any positive integer $n$,
every proper submodule of the $%
\mathbb{Z}
$-module $%
\mathbb{Z}
_{n}$ satisfies the $D$-annihilator condition. Indeed, suppose that
$rm\in\left\langle \bar{d}\right\rangle $ for some integer $d$ dividing $n$.
Put $n=cd$ then $cr\bar{m}=0$. Since $Ann_{M}(r)=0_{M}$, we get $c\bar{m}=0$
and so $\bar{m}\in\left\langle \bar{d}\right\rangle $.

\begin{proposition}
\label{eqM}Let $N$ be a proper submodule of an $R$-module $M$ satisfying the
$D$-annihilator condition. Then the following are equivalent.
\end{proposition}

\begin{enumerate}
\item $N$ is a semi $r$-submodule of $M.$

\item For $r\in R$ and a submodule $K$ of $M$ with $r^{2}K\subseteq N$ and
$Ann_{M}(r)=0_{M}$, then $rK\subseteq N$.
\end{enumerate}

\begin{proof}
(1)$\Rightarrow$(2). Suppose that $r^{2}K\subseteq N$ and $Ann_{M}%
(r)=0_{M}=Ann_{M}(r^{2})$. If $K\subseteq N$, then we are done. If
$K\nsubseteq N$, then $Ann_{R}(k)=0_{R}$ for each $k\in K$ since by assumption
$K\cap T(M)=\left\{  0_{M}\right\}  $. Since $N$ is a semi $r$-submodule, we
conclude that $rk\in N$. Therefore, $rk\in N$ for all $k\in K$ and the result follows.

(2)$\Rightarrow$(1). is straightforward.
\end{proof}

Recall that an $R$-module $M$ is called a multiplication module if every
submodule $N$ of $M$ has the form $IM$ for some ideal $I$ of $R$. Moreover, we
have $N=(N:_{R}M)M$. Next, we conclude a useful characterization for semi
$r$-submodules. First, recall the following lemmas.

\begin{lemma}
\label{Smith}\cite{Smith} Let $N$ be a submodule of a finitely generated
faithful multiplication $R$-module $M.$ For an ideal $I$ of $R$,
$(IN:_{R}M)=I(N:_{R}M)$, and in particular, $(IM:_{R}M)=I$.
\end{lemma}

\begin{lemma}
\cite{Majed}\label{Majed} Let $N$ is a submodule of faithful multiplication
$R$-module $M$. If $I$ is a finitely generated faithful multiplication ideal
of $R$, then

\begin{enumerate}
\item $N=(IN:_{M}I)$.

\item If $N\subseteq IM$, then $(JN:_{M}I)=J(N:_{M}I)$ for any ideal $J$ of
$R$.
\end{enumerate}
\end{lemma}

\begin{theorem}
\label{IM}Let $M$ be a finitely generated faithful multiplication $R$-module.
Then a submodule $N=IM$ satisfying the $D$-annihilator condition is a semi
$r$-submodule of $M$ if and only if $I$ is a semi $r$-ideal of $R$.
\end{theorem}

\begin{proof}
Suppose $N=IM$ is a semi $r$-submodule of $M$ and let $r\in R$ such that
$r^{2}\in I$ with $Ann_{R}(r)=0$. We claim that $Ann_{M}(r)=0_{M}$. Indeed, if
there is $0_{M}\neq m\in M$ such that $rm=0_{M}$, then $\left\langle
r\right\rangle (\left\langle m\right\rangle :_{R}M)=(\left\langle
rm\right\rangle :_{R}M)=(0_{M}:_{R}M)=0$ by Lemma \ref{Smith}. Thus,
$(\left\langle m\right\rangle :_{R}M)=0$ as $Ann_{R}(r)=0$ and then
$\left\langle m\right\rangle =(\left\langle m\right\rangle :_{R}M)M=0_{M}$, a
contradiction. Since $N$ satisfies the $D$-annihilator condition and
$r^{2}M\subseteq IM$, then $rM\subseteq IM$ by Proposition \ref{eqM}. Thus,
$r\in(rM:_{R}M)\subseteq(IM:_{R}M)=I$, as needed.

Conversely, suppose that $I$ is a semi $r$-ideal of $R$. Let $r\in R$ and
$K=JM$ be a submodule of $M$ such that $r^{2}JM=r^{2}K\subseteq IM$ and
$Ann_{M}(r)=0_{M}$. Take $A=rJ$ and note that $A^{2}\subseteq r^{2}%
JM:M\subseteq(IM:_{R}M)=I$ by Lemma \ref{Smith}. Now, we claim that $A\cap
zd(R)=\left\{  0\right\}  $. Suppose on contrary that there exists $0\neq
a=rj\in A$ such that $Ann_{R}(a)\neq0$. Choose $0\neq b\in R$ with $ab=rjb=0$.
Then $rjbM=0_{M}$ and so $jbM=0_{M}$ as $Ann_{M}(r)=0_{M}$. Since $b\neq0$,
$jM\subseteq K$ and $N$ satisfies the $D$-annihilator condition, then $jM=0$
and we conclude $j=0$ as $M$\textbf{ }is faithful, which is a
contradiction.\textbf{ }Therefore, $A\cap zd(R)=\left\{  0\right\}  $ and
$A\subseteq I$ by Corollary \ref{coro}. Thus, $rK=rJM=AM\subseteq IM=N$ as needed.
\end{proof}

In view of Theorem \ref{IM} we give the following characterization.

\begin{corollary}
\label{(N:M)}Let $R$ be a ring and $M$ be a finitely generated faithful
multiplication $R$-module. For a submodule $N$ of $M$ satisfying the
$D$-annihilator condition, the following statements are equivalent.

\begin{enumerate}
\item $N$ is a semi $r$-submodule of $M$.

\item $(N:_{R}M)$ is semi $r$-ideal of $R$.

\item $N=IM$ for some semi $r$-ideal $I$ of $R$.
\end{enumerate}
\end{corollary}

Let $N$ be a submodule of an $R$-module $M$ and $I$ be an ideal of $R$. The
residual of $N$ by $I$ is the set $(N:_{M}I)=\{m\in M:Im\subseteq N\}$. It is
clear that $(N:_{M}I)$ is a submodule of $M$ containing $N$. More generally,
for any subset $S\subseteq R$, $(N:_{M}S)$ is a submodule of $M$ containing
$N$. We recall that $M$-$rad(N)$ denotes the intersection of all prime
submodules of $M$ containing $N$. Moreover, if $M$ is finitely generated
faithful multiplication, then $M$-$rad(N)=\sqrt{(N:_{R}M)}M$, \cite{Smith}.

\begin{proposition}
Let $M$ be a finitely generated multiplication $R$-module and $N$ be a semi
$r$-submodule of $M$ satisfying the $D$-annihilator condition$.$
\end{proposition}

\begin{enumerate}
\item For any ideal $I$ of $R$ with $(N:_{M}I)\neq M$, $(N:_{M}I)$ is a semi
$r$-submodule of $M.$

\item If $M$ is faithful, then $(M$-$rad(N):_{R}M)\subseteq zd(R)\cup
\sqrt{(N:_{R}M)}$.
\end{enumerate}

\begin{proof}
(1) First, we show that $(N:_{M}I)$ satisfies the $D$-annihilator condition$.$
Let $K$ be a submodule of $M$ and $r\in R$ such that $rK\subseteq(N:_{M}I)$,
$K\nsubseteq(N:_{M}I)$ and $Ann_{M}(r)=0_{M}$. Then $rIK\subseteq N$ and so
$IK\cap T(M)=\left\{  0_{M}\right\}  $. It follows clearly that $K\cap
T(M)=\left\{  0_{M}\right\}  $ as needed. Suppose $N$ is a semi $r$-submodule
of $M$. Let $K$ be a submodule of $M$ such that $r^{2}K\subseteq(N:_{M}I)$ and
$Ann_{M}(r)=0_{M}$. Then $r^{2}IK\subseteq N$ which implies that $rIK\subseteq
N$ by Proposition \ref{eqM} and thus, $rK\subseteq(N:_{M}I)$. Therefore,
$(N:_{M}I)$ is a semi $r$-submodule of $M$ again by Proposition \ref{eqM}.

(2) Since $N$ be a semi $r$-submodule,$\ (N:_{R}M)$ is a semi $r$-ideal of $R$
by Corollary \ref{(N:M)}. Then the claim follows as $M$-$rad(N)=\sqrt
{(N:_{R}M)}M$ and by using Theorem \ref{char}(4).
\end{proof}

Next, we discuss when $IN$ is a semi $r$-submodule of a finitely generated
multiplication module $M$ where $I$ is an ideal of $R$ and $N$ is a submodule
of $M$. Recall that a submodule $N$ of an $R$-module $M$ is said to be pure if
$JN=JM\cap N$ for every ideal $J$ of $R$.

\begin{theorem}
\label{IN}Let $I$ be an ideal of a ring $R$, $M$ be a finitely generated
faithful multiplication $R$-module and $N$ be a submodule of $M$ such that
$IN$ satisfies the $D$-annihilator condition.
\end{theorem}

\begin{enumerate}
\item If $I$ is a semi $r$-ideal of $R$ and $N$ is a pure semi $r$-submodule
of $M$, then $IN$ is a semi $r$-submodule of $M$.

\item Let $I$ be a finitely generated faithful multiplication ideal of $R$. If
$IN$ is semi $r$-submodule of $M$, then either $I$ is a semi $r$-ideal of $R$
or $N$ is a semi $r$-submodule of $M$.
\end{enumerate}

\begin{proof}
(1) Suppose that $r^{2}K\subseteq IN$ and $Ann_{M}(r)=0_{M}$ for some $r\in R$
and a submodule $K=JM$ of $M$. If we take $A=rJ$, then $A^{2}\subseteq
r^{2}JM:M\subseteq(IN:M)=I(N:M)\subseteq I\cap(N:M)$. By Theorem \ref{IM},
$(N:_{R}M)$ is a semi $r$-ideal. We show that $A\cap zd(R)=\{0\}.$ Let $x\in
A\cap zd(R)$, say, $x=ry$ for some $y\in J$. Choose a nonzero $z\in R$ such
that $xz=ryz=0.$ Then $ryzM=0_{M}$ and since $Ann_{M}(r)=0_{M},$ we have
$yzM=0_{M}$. Since $M$ is faithful and $z\neq0$, we conclude that $yM=0_{M}$
and so $y=0$\textbf{. }Thus $x=0$, as required. Since $(N:_{R}M)$ is a semi
$r$-ideal, then $A\subseteq(N:_{R}M)$ by Corollary \ref{coro}. Therefore,
$rK=AM\subseteq(N:_{R}M)M=N$. On the other hand, since $I$ is also a semi
$r$-ideal, we have $A\subseteq I$ and so $rK=AM\subseteq IM$. Since $N$ is
pure, we conclude that $rK\subseteq IM\cap N=IN$ and we are done.

(2) First, by using Lemma \ref{Majed}, we note clearly that $N$ satisfies the
$D$-annihilator condition. We have two cases.

\textbf{Case I. }Let $N=M$. Then $I=I(N:_{R}M)=(IN:_{R}M)$ is a semi $r$-ideal
of $R$ by Corollary \ref{(N:M)}.

\textbf{Case II. }Let $N$ be proper. Observe that by Lemma \ref{Majed}, we
have the equality $(N:_{R}M)=((IN:_{M}I):_{R}M)=(I(N:_{R}M):_{M}I)$. Suppose
that $r^{2}\in(N:_{R}M)$ and $r\notin zd(R)$. Then $(rI)^{2}\subseteq
r^{2}I\subseteq I(N:_{R}M)=(IN:_{R}M)$ by Lemma \ref{Smith}. Here, similar to
the proof of Theorem \ref{IM}, it can be easily verify that $rI\cap
zd(R)=\{0\}$. Since $(IN:_{R}M)$ is a semi $r$-ideal, $rI\subseteq
(IN:_{R}M)=I(N:_{R}M)$ which means $r\in(I(N:_{R}M):_{M}I)=(N:_{R}M)$ by Lemma
\ref{Majed}. Thus, $(N:_{R}M)$ is a semi $r$-ideal of $R$ and Corollary
\ref{(N:M)} implies that $N$ is a semi $r$-submodule of $M$.
\end{proof}

Next, we study the behavior of the semi $r$-submodule property under module homomorphisms.

\begin{proposition}
\label{fsub}Let $M$ and $M^{\prime}$ be $R$-modules and $f:M\rightarrow
M^{\prime}$ be an $R$-module homomorphism.
\end{proposition}

\begin{enumerate}
\item If $f$ is an epimorphism and $N$ is a semi $r$-submodule of $M$ such
that $Ker(f)\subseteq N$ and $N\cap T(M)=\left\{  0_{M}\right\}  $, then
$f(N)$ is a semi $r$-submodule of $M^{\prime}$.

\item If $f$ is an isomorphism and $N^{\prime}$ is a semi $r$-submodule of
$M^{\prime}$, then $f^{-1}(N^{\prime})$ is a semi $r$-submodule of $M$.
\end{enumerate}

\begin{proof}
(1). Let $N$ be a semi $r$-submodule of $M$ and $r\in R$, $m^{\prime}:=f(m)\in
M^{\prime}$ $(m\in M)$ such that $r^{2}m^{\prime}\in f(N)$, $Ann_{M^{^{\prime
}\prime}}(r)=0_{M\text{ }^{\prime}}$ and $Ann_{R}(f(m))=0_{M\text{ }^{\prime}%
}$. Then $r^{2}m\in N$ as $Ker(f)\subseteq N$. We show that $Ann_{M}(r)=0_{M}%
$. If $r=0$, then the claim is obvious. Suppose $r\neq0$ and there is
$m_{1}\in M$ such that $rm_{1}=0_{M}$. Then $rf(m_{1})=0_{M^{\prime}}$ and so
$f(m_{1})=0_{M^{\prime}}$ as $Ann_{M^{^{\prime}\prime}}(r)=0_{M\text{
}^{\prime}}$. Thus, $m_{1}\in Ker(f)\cap T(M)\subseteq N\cap T(M)=\left\{
0_{M}\right\}  $ as needed. Also, it is clear that $Ann_{R}(m)=0_{M\text{ }}$.
Therefore, $rm\in N$ and so $rm^{\prime}\in f(N)$ as required.

(2). Let $N^{\prime}$ is a semi $r$-submodule of $M^{\prime}$. Suppose that
$r^{2}m\in f^{-1}(N^{\prime})$, $Ann_{M}(r)=0_{M}$ and $Ann_{R}(m)=0$ for some
$r\in R$ and $m\in M$. Then $r^{2}f(m)=f(r^{2}m)\in N^{\prime}$,
$Ann_{M^{\prime}}(r)=0_{M^{\prime}}$ and $Ann_{R}(f(m))=0$. Indeed, if
$rm^{\prime}=0$ for some $0\neq m^{\prime}=f(m_{1})\in M^{\prime}$, then
$rm_{1}\in K\operatorname{erf}=\{0_{M}\}$ and clearly $0\neq m_{1}\in M$, a
contradiction. Similarly, if there exists $0\neq c\in R$ such that
$cf(m)=0_{M^{\prime}}$, then $cm=0_{M}$ which is also a contradiction. Since
$N^{\prime}$ is a semi $R$-submodule, then $rf(m)\in N^{\prime}$ and so $rm\in
f^{-1}(N^{\prime})$. Thus, $f^{-1}(N^{\prime})$ is a semi $r$-submodule of
$M.$
\end{proof}

In the following, we discuss semi $r$-submodules of localizations of modules.
Here, the notation $Z_{N}(R)$ denotes the set $\{r\in R$: $rm\in N$ for some
$m\in M\backslash N\}.$

\begin{theorem}
\label{SM}Let $S$ be a multiplicatively closed subset of a ring $R$ and $M$ be
an $R$-module such that $S\cap Z(M)=\emptyset$.
\end{theorem}

\begin{enumerate}
\item If $N$ is a semi $r$-submodule of $M$ such that $(N:_{R}M)\cap
S=\emptyset$, then $S^{-1}N$ is a semi $r$-submodule of $S^{-1}M.$

\item If $S^{-1}N$ is a semi $r$-submodule of $S^{-1}R$ and $S\cap
Z_{N}(R)=\emptyset$, then $N$ is a semi $r$-submodule of $M.$
\end{enumerate}

\begin{proof}
(1)\ Let $\left(  \frac{r}{s}\right)  ^{2}\left(  \frac{m}{t}\right)  \in
S^{-1}N$ with $Ann_{S^{-1}M}(\frac{r}{s})=0_{S^{-1}M}$ and $Ann_{S^{-1}%
R}(\frac{m}{t})=0_{S^{-1}R}$ for some $\frac{r}{s}\in S^{-1}R$ and $\frac
{m}{t}\in S^{-1}M$. Choose $u\in S$ such that $r^{2}(um)\in N$. We show that
$Ann_{M}(r)=0_{M}$ and $Ann_{R}(um)=0.$ First, assume that $rm^{\prime}=0_{M}$
for some $m^{\prime}\in M.$ Then $\left(  \frac{r}{s}\right)  \left(
\frac{m^{\prime}}{1}\right)  =0_{S^{-1}M}$ and so $\frac{m^{\prime}}%
{1}=0_{S^{-1}M}$ as $Ann_{S^{-1}M}(\frac{r}{s})=0_{S^{-1}M}$. Hence, there
exists $v\in S$ such that $vm^{\prime}=0_{M}$. Since $S\cap Z(M)=\emptyset$,
then $m^{\prime}=0_{M}$ and so $Ann_{M}(r)=0_{M}$. Secondly, assume that
$r^{\prime}um=0$ for some $r^{\prime}\in R$. Then $\frac{r^{\prime}u}{1}%
\frac{m}{t}=0_{S^{-1}M}$ and $Ann_{S^{-1}R}(\frac{m}{t})=0_{S^{-1}R}$ imply
that $r^{\prime}us=0$ for some $s\in S$. But, clearly, $um\neq0_{M}$ and so
$us\in S\cap Z(M)=\emptyset$, a contradiction. Hence, $Ann_{R}(um)=0$.
Therefore, $r^{2}(um)\in N$ implies that $rum\in N$ and so $\frac{r}{s}%
\frac{m}{t}=\frac{rum}{sut}\in S^{-1}N$.

(2) Suppose that $r^{2}m\in N$ with $Ann_{M}(r)=0_{M}$ and $Ann_{R}(m)=0$ for
some $r\in R$ and $m\in M.$ Now, $\left(  \frac{r}{1}\right)  ^{2}\frac{m}%
{1}\in S^{-1}N$. If $Ann_{S^{-1}M}(\frac{r}{1})\neq0_{S^{-1}M}$, then there
exists $0_{S^{-1}M}\neq\frac{m^{\prime}}{t}\in S^{-1}M$ such that $\frac{r}%
{1}\frac{m^{\prime}}{t}=0_{S^{-1}M}$ which implies $urm^{\prime}=0_{M}$ for
some $u\in S.$ Since $Ann_{M}(r)=0_{M}$, we have $um^{\prime}=0_{M}$ and
$\frac{m^{\prime}}{t}=\frac{um^{\prime}}{ut}=0_{S^{-1}M}$, a contradiction.
Now, assume that $Ann_{S^{-1}R}(\frac{m}{1})\neq0_{S^{-1}R}.$ Then
$\frac{r^{\prime}}{s^{\prime}}\frac{m}{1}=0_{S^{-1}M}$ for some $0_{S^{-1}%
R}\neq\frac{r^{\prime}}{s^{\prime}}\in S^{-1}R$. Thus, $r^{\prime}vm=0$ for
some $v\in S$ and clearly $r^{\prime}m\neq0_{M}$. Hence, again $v\in S\cap
Z(M)=\emptyset$, a contradiction. Thus, $Ann_{S^{-1}M}(\frac{r}{1}%
)=0_{S^{-1}M}$ and $Ann_{S^{-1}R}(\frac{m}{1})=0_{S^{-1}R}$ imply that
$\frac{r}{1}\frac{m}{1}\in S^{-1}N$ and so $wrm\in N$ for some $w\in S$. Since
$S\cap Z_{N}(M)=\emptyset$, we conclude that $rm\in N$, as desired.
\end{proof}

We recall from \cite{Ande} that for an $R$-module $M$, we have
\[
zd(R(+)M)=\left\{  (r,m)|\text{ }r\in zd(R)\cup Z(M)\text{, }m\in M\right\}
\]
where $Z(M)=\left\{  r\in R:rm=0\text{ for some }0_{M}\neq m\in M\right\}  $.
In the following proposition, we justify the relation between semi $r$-ideals
of $R$ and those of the idealization ring $R(+)M$.

\begin{proposition}
\label{Ide}Let $M$ be an $R$-module and $I$ be a proper ideal of $R$.

\begin{enumerate}
\item If $I$ is a semi $r$-ideal of $R$, then $I(+)M$ is a semi $r$-ideal of
$R(+)M.$ Moreover, the converse is true if $Z(M)\subseteq zd(R)$.

\item If $I$ is a semi $r$-ideal of $R$ and $N$ is an $r$-submodule of $M$,
then $I(+)N$ is a semi $r$-ideal of $R(+)M$. Moreover, the converse is true if
$Z(M)\subseteq zd(R)$.
\end{enumerate}
\end{proposition}

\begin{proof}
(1). Suppose that $(a,m)^{2}\in I(+)M$ and $(a,m)\notin zd(R(+)M)$. Then
$a^{2}\in I$ and $a\notin zd(R)$. Since $I$ is a semi $r$-ideal, we conclude
that $a\in I$ and so $(a,m)\in I(+)M$. Now, assume that $Z(M)\subseteq zd(R)$
and $I(+)M$ is a semi $r$-ideal of $R(+)M$. Let $a\in R$ such that $a^{2}\in
I$ but $a\notin I$. Then $(a,0)^{2}\in I(+)M$ and $(a,0)\notin I(+)M$ which
imply that $(a,0)\in zd(R(+)M)$. Since $Z(M)\subseteq zd(R)$, we conclude that
$a\in zd(R)$ and we are done.

(2). Suppose that $(a,m)^{2}\in I(+)N$ and $(a,m)\notin zd(R(+)M)$. Then $a\in
I$ as in (1). Moreover, $a.m\in N$ as $IM\subseteq N$. Since also, $a\notin
Z(M)$, then $Ann_{M}(a)=0$. Therefore, $m\in N$ as $N$ is an $r$-submodule of
$M$ and $(a,m)\in I(+)N$ as needed. If $Z(M)\subseteq zd(R)$, then similar to
the proof of (1), the converse holds.
\end{proof}

\begin{remark}
In general, if $Z(M)\nsubseteq zd(R)$, then the converse of Proposition
\ref{Ide} need not be true. For example, consider the idealization ring $R=%
\mathbb{Z}
(+)%
\mathbb{Z}
_{4}$ and the ideal $4%
\mathbb{Z}
(+)%
\mathbb{Z}
_{4}$ of $R$. Let $(a,m)^{2}\in4%
\mathbb{Z}
(+)%
\mathbb{Z}
_{4}$ for $(a,m)\in R$. Then $a^{2}\in4%
\mathbb{Z}
$ and so $(a,m)\in2%
\mathbb{Z}
\times%
\mathbb{Z}
_{4}=zd(R)$. Thus, $4%
\mathbb{Z}
(+)%
\mathbb{Z}
_{4}$ is a (semi) $r$-ideal of $R$. On the other hand, $4%
\mathbb{Z}
$ is not a semi $r$-ideal of $%
\mathbb{Z}
$.
\end{remark}

\section{Semi $r$-submodules of amalgamated modules}

Let $R$ be a ring, $J$ an ideal of $R$ and $M$ an $R$-module. Recently, in
\cite{Bouba}, the duplication of the $R$-module $M$ along the ideal $J$
(denoted by $M\Join J$) is defined as%
\[
M\Join J=\left\{  (m,m^{\prime})\in M\times M:m-m^{\prime}\in JM\right\}
\]
which is an $(R\Join J)$-module with scaler multiplication defined by
$(r,r+j)\cdot(m,m^{\prime})=(rm,(r+j)m^{\prime})$ for $r\in R$, $j\in J$ and
$(m,m^{\prime})\in M\Join J$. For various properties and results concerning
this kind of modules, one may see \cite{Bouba}.

Let $J$ be an ideal of a ring $R$ and $N$ be a submodule of an $R$-module $M$. Then%

\[
N\Join J=\left\{  (n,m)\in N\times M:n-m\in JM\right\}
\]
and
\[
\bar{N}=\left\{  (m,n)\in M\times N:m-n\in JM\right\}
\]
are clearly submodules of $M\Join J$. Moreover,%

\[
Ann_{R\Join J}(M\Join J)=(r,r+j)\in R\Join I|r\in Ann_{R}(M)\text{ and }j\in
Ann_{R}(M)\cap J\}
\]
and so $M\Join J$ is a faithful $R\Join J$ -module if and only if $M$ is a
faithful $R$-module, \cite[Lemma 3.6]{Bouba}.

In general, let $f:R_{1}\rightarrow R_{2}$ be a ring homomorphism, $J$ be an
ideal of $R_{2}$, $M_{1}$ be an $R_{1}$-module, $M_{2}$ be an $R_{2}$-module
(which is an $R_{1}$-module induced naturally by $f$) and $\varphi
:M_{1}\rightarrow M_{2}$ be an $R_{1}$-module homomorphism. The subring
\[
R_{1}\Join^{f}J=\left\{  (r,f(r)+j):r\in R_{1}\text{, }j\in J\right\}
\]
of $R_{1}\times R_{2}$ is called the amalgamation of $R_{1}$ and $R_{2}$ along
$J$ with respect to $f$. In \cite{Rachida}, the amalgamation of $M_{1}$ and
$M_{2}$ along $J$ with respect to $\varphi$ is defined as%

\[
M_{1}\Join^{\varphi}JM_{2}=\left\{  (m_{1},\varphi(m_{1})+m_{2}):m_{1}\in
M_{1}\text{ and }m_{2}\in JM_{2}\right\}
\]
which is an $(R_{1}\Join^{f}J)$-module with the scaler product defined as
\[
(r,f(r)+j)(m_{1},\varphi(m_{1})+m_{2})=(rm_{1},\varphi(rm_{1})+f(r)m_{2}%
+j\varphi(m_{1})+jm_{2})
\]
For submodules $N_{1}$ and $N_{2}$ of $M_{1}$ and $M_{2}$, respectively, one
can easily justify that the sets
\[
N_{1}\Join^{\varphi}JM_{2}=\left\{  (m_{1},\varphi(m_{1})+m_{2})\in M_{1}%
\Join^{\varphi}JM_{2}:m_{1}\in N_{1}\right\}
\]
and
\[
\overline{N_{2}}^{\varphi}=\left\{  (m_{1},\varphi(m_{1})+m_{2})\in M_{1}%
\Join^{\varphi}JM_{2}:\text{ }\varphi(m_{1})+m_{2}\in N_{2}\right\}
\]
are submodules of $M_{1}\Join^{\varphi}JM_{2}$.

Note that if $R=R_{1}=R_{2}$, $M=M_{1}=M_{2}$, $f=Id_{R}$ and $\varphi=Id_{M}%
$, then the amalgamation of $M_{1}$ and $M_{2}$ along $J$ with respect to
$\varphi$ is exactly the duplication of the $R$-module $M$ along the ideal
$J$. Moreover, in this case, we have $N_{1}\Join^{\varphi}JM_{2}=N\Join J$ and
$\overline{N_{2}}^{\varphi}=\bar{N}$.

\begin{theorem}
\label{Amalg}Consider the $(R_{1}\Join^{f}J)$-module $M_{1}\Join^{\varphi
}JM_{2}$ defined as above. Assume $JM_{2}=\left\{  0_{M_{2}}\right\}  $ and
let $N_{1}$ be submodule of $M_{1}$. Then
\end{theorem}

\begin{enumerate}
\item $N_{1}$ is an $r$-submodule of $M_{1}$ if and only if $N_{1}%
\Join^{\varphi}JM_{2}$ is an $r$-submodule of $M_{1}\Join^{\varphi}JM_{2}$.

\item If $N_{1}$ is a semi $r$-submodule of $M_{1}$, then $N_{1}\Join
^{\varphi}JM_{2}$ is a semi $r$-submodule of $M_{1}\Join^{\varphi}JM_{2}$.

\item If $M_{2}$ is faithful and $N_{1}\Join^{\varphi}JM_{2}$ is a semi
$r$-submodule of $M_{1}\Join^{\varphi}JM_{2}$, then $N_{1}$ is a semi
$r$-submodule of $M_{1}$.
\end{enumerate}

\begin{proof}
(1) Let $N_{1}$ be an $r$-submodule of $M_{1}$ and let $(r_{1},f(r_{1})+j)\in
R_{1}\Join^{f}J$, $(m_{1},\varphi(m_{1}))\in M_{1}\Join^{\varphi}JM_{2}$ such
that $(r_{1},f(r_{1})+j)(m_{1},\varphi(m_{1}))\in N_{1}\Join^{\varphi}JM_{2}$
and $Ann_{M_{1}\Join^{\varphi}JM_{2}}((r_{1},f(r_{1})+j))=0_{M_{1}%
\Join^{\varphi}JM_{2}}$. Then $r_{1}m_{1}\in N_{1}$ and we prove that
$Ann_{M_{1}}(r_{1})=0_{M_{1}}$. Suppose $r_{1}m_{1}^{\prime}=0_{M_{1}}$ for
some $m_{1}^{\prime}\in M_{1}$. Then $(r_{1},f(r_{1})+j)(m_{1}^{\prime
},\varphi(m_{1}^{\prime}))=(0_{M_{1}},j\varphi(m_{1}^{\prime}))=(0_{M_{1}%
},0_{M_{2}})$ as $JM_{2}=\left\{  0_{M_{2}}\right\}  $. Thus, $(m_{1}^{\prime
},\varphi(m_{1}^{\prime}))\in Ann_{M_{1}\Join^{\varphi}JM_{2}}((r_{1}%
,f(r_{1})+j))=0_{M_{1}\Join^{\varphi}JM_{2}}$. Hence, $m_{1}^{\prime}%
=0_{M_{1}}$ and $Ann_{M_{1}}(r_{1})=0_{M_{1}}$. By assumption, $m_{1}\in
N_{1}$ and then $(m_{1},\varphi(m_{1}))\in N_{1}\Join^{\varphi}JM_{2}$, as needed.

Conversely, let $r_{1}\in R_{1}$ and $m_{1}\in M_{1}$ such that $r_{1}m_{1}\in
N_{1}$ and $Ann_{M_{1}}(r_{1})=0_{M_{1}}$. Then $(r_{1},f(r_{1}))\in
R_{1}\Join^{f}J$ , $(m_{1},\varphi(m_{1}))\in M_{1}\Join^{\varphi}JM_{2}$ and
$(r_{1},f(r_{1}))(m_{1},\varphi(m_{1}))=(r_{1}m_{1},\varphi(r_{1}m_{1}))\in
N_{1}\Join^{\varphi}JM_{2}$. Moreover, $Ann_{M_{1}\Join^{\varphi}JM_{2}%
}((r_{1},f(r_{1})))=0_{M_{1}\Join^{\varphi}JM_{2}}$. Indeed, suppose that
there $(m_{1}^{\prime},\varphi(m_{1}^{\prime}))\in M_{1}\Join^{\varphi}JM_{2}$
such that $(r_{1},f(r_{1}))(m_{1}^{\prime},\varphi(m_{1}^{\prime}%
))=0_{M_{1}\Join^{\varphi}JM_{2}}$. Then $(m_{1}^{\prime},\varphi
(m_{1}^{\prime}))=(0_{M_{1}},0_{M_{2}})$ as $Ann_{M_{1}}(r_{1})=0_{M_{1}}$.
Since $N_{1}\Join^{\varphi}JM_{2}$ is an $r$-submodule of $M_{1}\Join
^{\varphi}JM_{2}$, then $(m_{1},\varphi(m_{1}))\in N_{1}\Join^{\varphi}JM_{2}$
so that $m_{1}\in N_{1}$ and we are done.

(2) Let $(r_{1},f(r_{1})+j)\in R_{1}\Join^{f}J$ and $(m_{1},\varphi(m_{1}))\in
M_{1}\Join^{\varphi}JM_{2}$ such that $(r_{1},f(r_{1})+j)^{2}(m_{1}%
,\varphi(m_{1}))\in N_{1}\Join^{\varphi}JM_{2}$, $Ann_{M_{1}\Join^{\varphi
}JM_{2}}((r_{1},f(r_{1})+j))=0_{M_{1}\Join^{\varphi}JM_{2}}$ and
$Ann_{R_{1}\Join^{f}J}((m_{1},\varphi(m_{1})))=0_{R_{1}\Join^{f}J}$. Then
$r_{1}^{2}m_{1}\in N_{1}$ and similar to the proof of (1), we have
$Ann_{M_{1}}(r_{1})=0_{M_{1}}$. We show that $Ann_{R_{1}}(m_{1})=0_{R_{1}}.$
Assume on the contrary that there is nonzero element $r_{1}\in R_{1}$ such
that $r_{1}m_{1}=0_{R_{1}}.$ Then, $(r_{1},f(r_{1}))(m_{1},\varphi
(m_{1}))=0_{M_{1}\Join^{\varphi}JM_{2}}$, but our assumption $Ann_{R_{1}%
\Join^{f}J}((m_{1},\varphi(m_{1})))=0_{R_{1}\Join^{f}J}$ implies that
$(r_{1},f(r_{1}))=0_{R_{1}\Join^{f}J}$; i.e. $r_{1}=0_{R_{1}}$, a
contradiction. Thus $Ann_{R_{1}}(m_{1})=0_{R_{1}},$ and it follows that
$r_{1}m_{1}\in N_{1}$ and so $(r_{1},f(r_{1})+j)(m_{1},\varphi(m_{1}%
)+m_{2})\in N_{1}\Join^{\varphi}JM_{2}$.

(3) Since $M_{2}$ is faithful, then clearly $J=\left\{  0_{R_{2}}\right\}  $.
Let $r_{1}\in R_{1}$ and $m_{1}\in M_{1}$ such that $r_{1}^{2}m_{1}\in N_{1}$,
$Ann_{M_{1}}(r_{1})=0_{M_{1}}$ and $Ann_{R_{1}}(m_{1})=0_{R_{1}}$. Then
$(r_{1},f(r_{1}))^{2}(m_{1},\varphi(m_{1}))\in N_{1}\Join^{\varphi}JM_{2}$
where $(r_{1},f(r_{1}))\in R_{1}\Join^{f}J$ and $(m_{1},\varphi(m_{1}))\in
M_{1}\Join^{\varphi}JM_{2}$. Again, similar to the proof of (1), we have
$Ann_{M_{1}\Join^{\varphi}JM_{2}}((r_{1},f(r_{1})))=0_{M_{1}\Join^{\varphi
}JM_{2}}$. Moreover, suppose there is $(r_{1}^{\prime},f(r_{1}^{\prime}))\in
R_{1}\Join^{f}J$ such that $(r_{1}^{\prime}m_{1},\varphi(r_{1}^{\prime}%
m_{1}))=(r_{1}^{\prime},f(r_{1}^{\prime})+j)(m_{1},\varphi(m_{1}%
))=0_{M_{1}\Join^{\varphi}JM_{2}}$. Then $(r_{1}^{\prime},f(r_{1}^{\prime
}))=(0_{R_{1}},0_{R_{2}})$ as $Ann_{R_{1}}(m_{1})=0_{R_{1}}$ and so
$Ann_{R_{1}\Join^{f}J}((m_{1},\varphi(m_{1})))=0_{M_{1}\Join^{\varphi}JM_{2}}%
$. By assumption, $(r_{1},f(r_{1}))(m_{1},\varphi(m_{1}))\in N_{1}%
\Join^{\varphi}JM_{2}$. It follows that $r_{1}m_{1}\in N_{1}$ and $N_{1}$ is a
semi $r$-submodule of $M_{1}$.
\end{proof}

\begin{corollary}
\label{Dup1}Let $N$ be a submodule of an $R$-module $M$ and $J$ be an ideal of
$R$. Then
\end{corollary}

\begin{enumerate}
\item If $N\Join J$ is an $r$-submodule of $M\Join J$, then $N$ is an
$r$-submodule of $M$. The converse is true if $JM=0_{M}$.

\item If $N\Join J$ is a semi $r$-submodule of $M\Join J$, then $N$ is a semi
$r$-submodule of $M$. The converse is true if $JM=0_{M}$.
\end{enumerate}

\begin{proof}
(1) Let $r\in R$ and $m\in M$ such that $rm\in N$ and $Ann_{M}(r)=0_{M}$. Then
$(r,r)(m,m)\in N\Join J$ and clearly, $Ann_{M\Join J}((r,r))=0_{M\Join J}$.
Thus, $(m,m)\in N\Join J$ and so $m\in N$ as needed. Conversely, suppose
$JM=0_{M}$ and let $(r,r+j)\in R\Join J$, $(m,m+m^{\prime})\in M\Join J$ such
that $(r,r+j)(m,m+m^{\prime})\in N\Join J$ and $Ann_{M\Join J}%
((r,r+j))=0_{M\Join J}$. If $rm^{\prime\prime}=0_{M}$ for some $m^{\prime
\prime}\in M$, then $(r,r+j)(m^{\prime\prime},m^{\prime\prime})=(0,jm^{\prime
\prime})=(0_{M},0_{M})$ as $JM=0_{M}$. Thus, $m^{\prime\prime}=0_{M}$ and
$Ann_{M}(r)=0_{M}$. Since $rm\in N$, then $m\in N$ and so $(m,m+m^{\prime})\in
N\Join J$.

(2) Let $r\in R$ and $m\in M$ such that $r^{2}m\in N$, $Ann_{M}(r)=0_{M}$ and
$Ann_{R}(m)=0_{R}$. Then $(r,r)^{2}(m,m)\in N\Join J$. If there exists an
element $(m^{\prime},m^{\prime\prime})$ of $M\Join J$, $(r,r)(m^{\prime
},m^{\prime\prime})=(0_{M},0_{M})$, then clearly $(m^{\prime},m^{\prime\prime
})=(0_{M},0_{M})$ as $Ann_{M}(r)=0_{M}$; and so $Ann_{M\Join J}%
((r,r))=0_{M\Join J}$. Also, if for $(r^{\prime},r^{\prime}+j)\in R\Join J$,
$(r^{\prime},r^{\prime}+j)(m,m)=(0_{M},0_{M})$, then $(r^{\prime},r^{\prime
}+j)=(0_{R},0_{R})$ and $Ann_{R\Join J}((m,m))=0_{R\Join J}$. By assumption,
$(r,r)(m,m)\in N\Join J$ and so $rm\in N$. The proof of the converse part is
similar to that of the converse of (1).
\end{proof}

\begin{theorem}
\label{Amalg2}Consider the $(R_{1}\Join^{f}J)$-module $M_{1}\Join^{\varphi
}JM_{2}$ defined as in Theorem \ref{Amalg} and let $N_{2}$ be a submodule of
$M_{2}$.
\end{theorem}

\begin{enumerate}
\item If $N_{2}$ is an $r$-submodule of $M_{2}$, $JM_{2}\neq\left\{  0_{M_{2}%
}\right\}  $ and $T(M_{2})\subseteq JM_{2}$, then $\overline{N_{2}}^{\varphi}$
is an $r$-submodule of $M_{1}\Join^{\varphi}JM_{2}$. Moreover, if $f$ is an
epimorphism and $\varphi$ is an isomorphism, then the converse holds.

\item If $f$ and $\varphi$ are isomorphisms and $\overline{N_{2}}^{\varphi}$
is a semi $r$-submodule of $M_{1}\Join^{\varphi}JM_{2}$, then $N_{2}$ is a
semi $r$-submodule of $M_{2}$.
\end{enumerate}

\begin{proof}
(1). Suppose $N_{2}$ is an $r$-submodule of $M_{2}$. Let $(r_{1}%
,f(r_{1})+j)\in R_{1}\Join^{f}J$ and $(m_{1},\varphi(m_{1})+m_{2})\in
M_{1}\Join JM_{2}$ such that $(r_{1},f(r_{1})+j)(m_{1},\varphi(m_{1}%
)+m_{2})\in\overline{N_{2}}^{\varphi}$ and $Ann_{M_{1}\Join^{\varphi}JM_{2}%
}((r_{1},f(r_{1})+j))=0_{M_{1}\Join^{\varphi}JM_{2}}$. Then $(f(r_{1}%
)+j)(\varphi(m_{1})+m_{2})\in N_{2}$ and $Ann_{M_{2}}((f(r_{1})+j))=0_{M_{2}}%
$. Indeed, suppose $(f(r_{1})+j)m_{2}^{\prime}=0_{M_{2}}$ for some $0_{M_{2}%
}\neq m_{2}^{\prime}\in M_{2}$. If $m_{2}^{\prime}\in JM_{2}$, then
$(r_{1},f(r_{1})+j)(0_{M_{1}},0_{M_{2}}+m_{2}^{\prime})=0_{M_{1}\Join JM_{2}}$
where $(0_{M_{1}},0_{M_{2}}+m_{2}^{\prime})\neq0_{M_{1}\Join JM_{2}}$, a
contradiction. If $m_{2}^{\prime}\notin JM_{2}$, then $m_{2}^{\prime}\notin
T(M_{2})$ and so $(f(r_{1})+j)=0_{R_{2}}$. If we choose $0\neq m_{2}%
^{\prime\prime}\in JM_{2}$, then $(r_{1},f(r_{1})+j)(0_{M_{1}},m_{2}%
^{\prime\prime})=0_{M_{1}\Join JM_{2}}$ which is also a contradiction. By
assumption, $\varphi(m_{1})+m_{2})\in N_{2}$ and so $(m_{1},\varphi
(m_{1})+m_{2})\in\overline{N_{2}}^{\varphi}$.

Conversely, suppose $\varphi$ is an isomorphism and $\overline{N_{2}}%
^{\varphi}$ is an $r$-submodule of $M_{1}\Join^{\varphi}JM_{2}$. Let
$r_{2}=f(r_{1})\in R_{2}$ and $m_{2}=\varphi(m_{1})\in M_{2}$ such that
$r_{2}m_{2}\in N_{2}$ and $Ann_{M_{2}}(r_{2})=0_{M_{2}}$. Then $(r_{1}%
,r_{2})\in R_{1}\Join^{f}J$, $(m_{1},m_{2})\in M_{1}\Join^{\varphi}JM_{2}$ and
$(r_{1},r_{2})(m_{1},m_{2})\in\overline{N_{2}}^{\varphi}$. Suppose on contrary
that there is $(m_{1}^{\prime},\varphi(m_{1}^{\prime})+m_{2}^{\prime}%
)\neq0_{M_{1}\Join^{\varphi}JM_{2}}$ such that $(r_{1},r_{2})(m_{1}^{\prime
},\varphi(m_{1}^{\prime})+m_{2}^{\prime})=0_{M_{1}\Join^{\varphi}JM_{2}}$. If
$\varphi(m_{1}^{\prime})+m_{2}^{\prime}\neq0_{M_{2}}$, we get a contradiction.
If $\varphi(m_{1}^{\prime})+m_{2}^{\prime}=0_{M_{2}}$ (and so $m_{1}^{\prime
}\neq0_{M_{1}}$), then clearly $r_{2}m_{2}^{\prime}=0_{M_{2}}$ and then
$m_{2}^{\prime}=0_{M_{2}}$. It follows that $\varphi(m_{1}^{\prime})=0_{M_{2}%
}$ and so $m_{1}^{\prime}=0_{M_{1}}$, a contradiction. Since $\overline{N_{2}%
}^{\varphi}$ is an $r$-submodule of $M_{1}\Join^{\varphi}JM_{2}$, then
$(m_{1},m_{2})\in\overline{N_{2}}^{\varphi}$ and so $m_{2}\in N_{2}$ as required.

(3) Let $r_{2}=f(r_{1})\in R_{2}$ and $m_{2}=\varphi(m_{1})\in M_{2}$ such
that $r_{2}^{2}m_{2}\in N_{2}$, $Ann_{M_{2}}(r_{2})=0_{M_{2}}$ and
$Ann_{R_{2}}(m_{2})=0_{R_{2}}$. Then $(r_{1},r_{2}))^{2}(m_{1},m_{2}%
)\in\overline{N_{2}}^{\varphi}$ where $(r_{1},f(r_{1}))\in R_{1}\Join^{f}J$
and $(m_{1},\varphi(m_{1}))\in M_{1}\Join^{\varphi}JM_{2}$. Similar to the
proof of the converse part of (1), we have $Ann_{M_{1}\Join^{\varphi}JM_{2}%
}((r_{1},r_{2}))=0_{M_{1}\Join^{\varphi}JM_{2}}$. We prove that $Ann_{R_{1}%
\Join^{f}J}((m_{1},m_{2}))=0_{R_{1}\Join^{f}J}$. Let $(r_{1}^{\prime}%
,f(r_{1}^{\prime})+j^{\prime})\in R_{1}\Join^{f}J$ such that $(r_{1}^{\prime
},f(r_{1}^{\prime})+j^{\prime})(m_{1},m_{2})=0_{M_{1}\Join^{\varphi}JM_{2}}$.
Then $f(r_{1}^{\prime})+j^{\prime}=0_{R_{2}}$ and $r_{1}^{\prime}%
m_{1}=0_{M_{1}}$. Thus, $f(r_{1}^{\prime})m_{2}=0$ and so $f(r_{1}^{\prime
})=0_{R_{2}}$. Since $f$ is one to one, then $r_{1}^{\prime}=0_{R_{1}}$ and so
$(r_{1}^{\prime},f(r_{1}^{\prime})+j^{\prime})=0_{R_{1}\Join^{f}J}$ as needed.
By assumption, $(r_{1},r_{2}))(m_{1},m_{2})\in\overline{N_{2}}^{\varphi}$ and
so $r_{2}m_{2}\in N_{2}$.
\end{proof}

\begin{corollary}
\label{Dup2}Let $N$ be a submodule of an $R$-module $M$ and $J$ be an ideal of
$R$. Then
\end{corollary}

\begin{enumerate}
\item If $\bar{N}$ is an $r$-submodule of $M\Join J$, then $N$ is an
$r$-submodule of $M$. The converse is true if $JM=0_{M}$.

\item If $\bar{N}$ is a semi $r$-submodule of $M\Join J$, then $N$ is a semi
$r$-submodule of $M$. The converse is true if $JM=0_{M}$.
\end{enumerate}

\begin{proof}
The proof is similar to that of Corollary \ref{Dup1} and left to the reader.
\end{proof}

\textbf{Statements \& Declarations}

The authors declare that no funds, grants, or other support were received
during the preparation of this manuscript. The authors have no relevant
financial or non-financial interests to disclose. All authors read and
approved the final manuscript.,

\end{document}